\documentclass[letterpaper, 11pt, reqno]{amsart}
\usepackage[top = 1in, bottom = 1in, left=1in, right=1in]{geometry}
\usepackage{setspace}
\usepackage[utf8]{inputenc}
\usepackage[T1]{fontenc}
\usepackage{lineno}  
\usepackage{tikz}
\usepackage{xcolor}
\usetikzlibrary{calc}
\newcounter{propcounter}
\usepackage[inline]{enumitem}
\setlist[enumerate,1]{label={\upshape (\roman*)}}

\usepackage{amssymb,amsmath,amsthm,graphicx,xcolor,mathtools,mathrsfs}
\usepackage{comment}

\usepackage{hyperref}
\hypersetup{colorlinks, linkcolor={red!50!black}, citecolor={green!50!black}, urlcolor={blue!50!black}}

\newcommand{\incup}{\mathbin{\mathchoice
{\ooalign{\hidewidth$\displaystyle\cdot$\hidewidth\cr$\displaystyle\cup$}}
{\ooalign{\hidewidth$\textstyle\cdot$\hidewidth\cr$\textstyle\cup$}}
{\ooalign{\hidewidth$\scriptstyle\cdot$\hidewidth\cr$\scriptstyle\cup$}}
{\ooalign{\hidewidth$\scriptscriptstyle\cdot$\hidewidth\cr$\scriptscriptstyle\cup$}}}}

\newcommand{\ca}{\mathcal{A}}
\newcommand{\cc}{\mathcal{C}}

\newcommand{\cp}{\mathcal{P}}

\newcommand{\piu}{\pi_{\rm u}}
\newcommand{\ex}{{\rm ex}}

\usepackage[nameinlink]{cleveref}
\theoremstyle{plain}
\newtheorem{theorem}{Theorem}[section]
\crefname{theorem}{Theorem}{Theorems}

\newtheorem{proposition}[theorem]{Proposition}
\crefname{proposition}{Proposition}{Propositions}

\crefname{corollary}{Corollary}{Corollaries}

\newtheorem{lemma}[theorem]{Lemma}
\crefname{lemma}{Lemma}{Lemmas}

\crefname{conjecture}{Conjecture}{Conjectures}

\crefname{problem}{Problem}{Problem}

\newtheorem{claim}[theorem]{Claim}
\crefname{claim}{Claim}{Claims}

\crefname{observation}{Observation}{Observations}

\crefname{setup}{Setup}{Setups}

\crefname{myth}{Myth}{Myths}

\crefname{fact}{Fact}{Facts}

\crefname{algorithm}{Algorithm}{Algorithms}

\crefname{remark}{Remark}{Remarks}

\crefname{example}{Example}{Examples}

\theoremstyle{definition}
\newtheorem{definition}[theorem]{Definition}
\crefname{definition}{Definition}{Definitions}

\crefname{construction}{Construction}{Constructions}

\newtheorem{question}[theorem]{Question}
\crefname{question}{Question}{Questions}

\numberwithin{equation}{section}

\crefformat{section}{\S#2#1#3}
\crefformat{subsection}{\S#2#1#3}
\crefformat{subsubsection}{\S#2#1#3}
\crefrangeformat{section}{\S\S#3#1#4 to~#5#2#6}
\crefmultiformat{section}{\S\S#2#1#3}{ and~#2#1#3}{, #2#1#3}{ and~#2#1#3}

\allowdisplaybreaks

\begin{document}
%\pagewiselinenumbers

\title[Extremal problems in uniformly dense
hypergraphs and digraphs]{Extremal problems in uniformly dense
hypergraphs and digraphs}
\author[H. Lin]{Hao Lin}
\address{School of Mathematical Sciences, University of Science and Technology of China, Hefei, China}
\email{baronlin001@gmail.com}
\author[G. Wang]{Guanghui Wang}
\author[W. Zhou]{Wenling Zhou}
\address{School of Mathematics, Shandong University, Jinan,  China.}
\email{\{\,ghwang\,|\,gracezhou\,\}@sdu.edu.cn}
\author[Y. Zhou]{Yiming Zhou}
\address{School of Mathematics, Shandong University, Jinan,  China.}
\email{202415555@mail.sdu.edu.cn}
%\thanks{G. Wang was supported by the National Key R\&D Program of China (2020YFA0712400) and the Natural Science Foundation of China (12231018).
%W. Zhou was supported by Natural Science Foundation of China  (12401457), the China Postdoctoral Science Foundation (2024M761780), the Natural Science Foundation of Shandong Province (ZR2024QA067) and Young Talent of Lifting engineering for Science and Technology in Shandong, China (SDAST2025QTA074).}

\begin{abstract}
The \emph{uniform Tur\'an density} $\pi_{\rm u}(F)$ of a $3$-uniform hypergraph (or $3$-graph) $F$ is the supremum of all $d$ such that there exist infinitely many $F$-free $3$-graphs $H$ in which every induced subhypergraph on a linearly sized vertex set has edge density at least $d$. 
Determining $\pi_{\rm u}(F)$ for a given $3$-graph $F$ was proposed by Erd\H{o}s and S\'os in the 1980s, yet only a few cases are known. In particular, it remains open whether $1/2$ can occur as a value of $\pi_{\rm u}$.

In this paper, we establish a novel connection between Tur\'an-type extremal problems for digraphs and uniform Tur\'an densities of $3$-graphs. Using digraph extremal results, we give the first verifiable conditions for $3$-graphs $F$ with $\pi_{\rm u}(F) = (r-1)/r$ and $\pi_{\rm u}(F) = (r-1)^2/r^2$ for all $r \ge 2$, and identify the corresponding $3$-graphs. In particular,  these $3$-graph classes contain some specific $3$-graphs, such as $K^{(3)-}_4$. 
We also present a sufficient condition ensuring $\pi_{\rm u}(F)=4/27$ and construct $3$-graphs satisfying it; in particular, our examples are different from the tight $3$-uniform cycles whose uniform Tur\'an density $4/27$ was determined in [\emph{Trans.\ Amer.\ Math.\ Soc.\ 376 (2023), 4765–4809}].
Finally, we give a short proof of the existence of $3$-graphs $F$ with $\pi_{\rm u}(F)=1/27$, 
originally established by Garbe, Kr\'al' and Lamaison [\emph{Israel J.\ Math.\ 259 (2024), 701–726}] via the hypergraph regularity method.
%Notably, this excludes the tight $3$-uniform cycles studied in [\emph{Trans.\ Amer.\ Math.\ Soc.\ 376 (2023), 4765–4809}], which also have uniform Tur\'an density $4/27$.
\end{abstract}

\maketitle
\thispagestyle{empty}
\vspace{-2.15em}

\section{Introduction}
\subsection{Tur\'an problems in graphs and hypergraphs}
Let $k \ge 2$ be an integer. A classical extremal problem for $k$-uniform hypergraphs (or $k$-graphs), introduced by Tur\'an~\cite{Tu-graph} about eighty years ago, asks to determine for a given $k$-graph $F$ its \textit{Tur\'an number} ${\rm ex}(n, F)$, the maximum number of edges in an $F$-free $k$-graph on $n$ vertices.
Tur\'an's theorem~\cite{Tu-graph} determines ${\rm ex}(n,K_r)$ for all complete graphs $K_r$ on $r$ vertices.
However, it is a long-standing open problem in extremal combinatorics to develop a general understanding of these numbers for hypergraphs. Ideally, one would like to determine them exactly, but even asymptotic results (for $k \ge 3$) are currently known only in certain cases; see the excellent survey by Keevash~\cite{Keevash-survey} for more details.
It is well known and not hard to observe that the sequence ${\rm ex}(n, F)/\binom{n}{k}$ is decreasing in $n$.
Thus, 
one often focuses on the \textit{Tur\'an density} $\pi(F)$ of $F$ defined by
$
\pi(F):=\lim_{n\to \infty}{{\rm ex}(n,F)}/{\binom{n}{k}}.
$
Tur\'an densities are well understood for graphs, i.e., $2$-graphs, see~\cite{E-Stone,E-Simonovits}.
In particular, let $\Pi^{(k)}=\{\pi(F): F \text{~is~a~} k\text{-graph}\}$. Erdős-Stone-Simonovits theorem~\cite{E-Stone, E-Simonovits} implies that
\[
\Pi^{(2)}=\{(r-1)/r: r\in \mathbb N_+\}.
\]
However, the extremal problems for hypergraphs are notoriously difficult, even for simple hypergraphs like the complete $3$-uniform  hypergraph $K^{(3)}_4$ on four vertices and $K^{(3)-}_4$, obtained by removing an edge from $K^{(3)}_4$.
Despite extensive efforts and attempts so far, the values of $\pi(K^{(3)}_4)$ and $\pi(K^{(3)-}_4)$ remain unresolved, see~\cite{k43-upbound, k43lowbound, k43-lowbound}.
Moreover, Conlon and Sch\"ulke~\cite{conlon2024hypergraphs} recently showed that $1/2$ is an accumulation point for the values in $\Pi^{(3)}$, yet they noted that, ``highlighting the depth of our ignorance about Tur\'an densities, it is not known if $1/2 \in  \Pi^{(3)}$.''

\subsection{Tur\'an problems in uniformly dense hypergraphs}
Most known and conjectured extremal constructions for Tur\'an-type problems  contain large independent sets, i.e., vertex sets of linear size spanning no edges. 
This led Erd\H{o}s and S\'os~\cite{E-Sos, Erdos-1990} in the 1980s
to propose a variant of this problem, restricting attention to $F$-free $3$-graphs that are uniformly dense on large subsets of the vertices.
Formally, given $d \in [0, 1]$ and $\mu>0$, a $3$-graph $H=(V, E)$ is said to be {\it uniformly $(d, \mu)$-dense} if for any $U\subseteq V$, we have $|\binom{U}{3}\cap E|\ge d\binom{|U|}{3}-\mu|V|^3$. The \emph{uniform Tur\'an density} $\pi_{{\rm u}}(F)$ for a given
$3$-graph $F$ is defined as
\[
\begin{split}
\pi_{{\rm u}}(F) = \sup \{ d\in [0,1] &: \text{for\ every\ } \mu>0 \ \text{and\ } n_0\in \mathbb{N},\ \text{there\ exists\ an\ } F \text{-free,~uniformly} \\
&\quad (d,\mu)\text{-dense}~  \text{$3$-graph~} H\ \text{with~} |V(H)|\geq n_0 \}.
\end{split}
\]
With this notation, Erd\H{o}s and S\'os  asked to determine $\pi_{{\rm u}}(K^{(3)-}_4)$ and $\pi_{{\rm u}}(K^{(3)}_4)$. 
However, determining $\pi_{{\rm u}}(F)$  of a given $3$-graph $F$ is also very challenging. The value $\pi_{\rm u}(K^{(3)-}_4)=1/4$ was recently determined by Glebov, Kr\'al' and Volec~\cite{k43minus-1}, and independently by Reiher, R\"odl and Schacht~\cite{k43minus-2}.
The conjecture that $\pi_{{\rm u}}(K^{(3)}_4)=1/2$ has been an urgent problem in this area since R\"odl~\cite{k43-rodl} gave a quasi-random construction in 1986. 
Beyond $K^{(3)-}_4$, Reiher, R\"odl and
Schacht~\cite{vanishing} characterised
all $3$-graphs $F$ with $\pi_{{\rm u}}(F)=0$. Let $\Pi_{\rm u}^{(3)}=\{\pi_{{\rm u}}(F): F \text{~is~a~} 3\text{-graph}\}$.
As a consequence of this characterization, $\Pi_{\rm u}^{(3)}$ is disjoint with the real interval $(0, 1/27)$.
Subsequently, using the hypergraph regularity method, Garbe, Kr\'al' and Lamaison~\cite{1/27}, Li and the first three authors~\cite{Hypergraph-llwzhou23}, Garbe,  Il'kovi{\v c}, Kr\'al',  Ku{\v c}er\'ak and Lamaison~\cite{8/27} constructed some families of $3$-graphs $F$ with $\pi_{{\rm u}}(F)=1/27$, $1/4$ and $8/27$, respectively.
Moreover, Buci{\'c}, Cooper, Kr\'al', Mohr and Munh\'a Correia~\cite{cycle-uniform} determined the uniform Tur\'an density of $3$-uniform tight cycles $C^{(3)}_\ell$ of length $\ell\ge 5$ (the uniform Tur\'an density is equal to $4/27$ if $3\nmid \ell$, and it is zero otherwise); 
Chen and Sch{\"u}lke~\cite{chen2022beyond} proved that $\pi_{\rm u}(F^{\star}_5)=1/4$, where $F^{\star}_5$ is the $3$-graph on five vertices obtained from $K^{(3)-}_4$ by adding a new vertex whose link graph is a matching of size $2$ on the vertices of $K^{(3)-}_4$.
Very recently, Lamaison~\cite{Ander} has provided a new approach to determining uniform Tur\'an densities that bypasses the use of the hypergraph regularity method.
To describe this approach, we need the notion of palettes, introduced by Reiher~\cite{reiher2020extremal} as a generalization of the construction of R\"odl~\cite{k43-rodl}.

\begin{definition}\label{def:palette}
A \emph{palette} $\mathcal P$ is a pair $(\mathcal C,\mathcal A)$, where $\mathcal C$ is a finite set, whose elements we call \emph{colors}, and $\mathcal A\subseteq \mathcal C^3$ is a set of ordered triples of colors, whose elements we call \emph{admissible triples}. In particular, for an admissible triple $(x, y, z)\in \mathcal A$, we refer to the
color $x$ as the {\it left} color, $y$ as the {\it top} color and $z$ as the {\it right} color. The \emph{density} of $\mathcal P$ is
\[
d(\mathcal P):={|\mathcal A|}/{|\mathcal C|^3}.
\]
Given a palette $\mathcal P=(\mathcal C,\mathcal A)$, we say that a $3$-graph $F$ is \emph{$\mathcal P$-colorable} if there exist a linear order $\prec$ on $V(F)$ and a function $\varphi:\partial F \to \mathcal C$ of the pairs of vertices covered by edges of $F$ such that every edge $\{u,v,w\}\in E(F)$ with $u\prec v\prec w$ satisfies $(\varphi(uv),\varphi(uw),\varphi(vw))\in \mathcal A$. 
\end{definition}

Palettes can be used to obtain lower bounds on uniform Tur\'an densities. Specifically, if $F$ is not $\mathcal P$-colorable for some palette $\mathcal P$, then $\pi_{\rm u}(F)\ge d(\mathcal P)$. The reason is that such $\mathcal P$ can be used to construct $F$-free uniformly $(d(\mathcal P),o(1))$-dense $3$-graphs. For example,  Rödl~\cite{k43-rodl} (see also~\cite{reiher2020extremal, k43minus-2}) gave a palette for each $r\ge2$,
\begin{equation}\label{eq:low bound palette}
\mathcal Q_r = ([r], \mathcal A(Q_r)), \text{~where~}  
\mathcal A (Q_r) = \{ (x, y, z) \in [r]^3 : x \ne y \}.
\end{equation}
Note that $d(\mathcal Q_r)=(r-1)/r$. 
Moreover, 
for each complete hypergraph $K^{(3)}_{r+2}$ on the vertex set $V:=\{1, 2, \dots, r+2\}$ and a function $\varphi: \binom{V}{2} \to [r]$, two of the $r+1$ pairs $\{1, i\}$ with $i\in \{2, 3,\dots, r+2\}$
containing the vertex $1$ must have the same color in $\varphi$,  which implies that $K^{(3)}_{r+2}$ is not $\mathcal Q_r$-colorable. Hence, we have
\begin{equation}\label{eq:complete hypergraph}
\pi_{\rm u}(K^{(3)}_{r+2})\ge \frac{r-1}{r}.
\end{equation}

It is natural to consider the \emph{palette Tur\'an density} $\pi^{\rm pal}_{\rm u}(F)$ of a given $3$-graph $F$, defined by
\[
\pi^{\rm pal}_{\rm u}(F)
:= \sup \{ d(\mathcal P) : \mathcal P \text{ is a palette and } F \text{ is not } \mathcal P\text{-colorable} \}.
\]
Lamaison~\cite{Ander} proved the following equivalence.

\begin{theorem}[\cite{Ander}]\label{Thm:ander}
For every $3$-graph $F$, we have $\pi_{\rm u}(F)=\pi^{\rm pal}_{\rm u}(F)$.
\end{theorem}

Using~\cref{Thm:ander}, Lamaison~\cite{Ander} proved that $\left\{\frac{1}{2}-\frac{1}{2r}: r\ge 2\right\}\subset \Pi_{\rm u}^{(3)}$, which implies that $1/2$ is an accumulation point for the values in $\Pi_{\rm u}^{(3)}$. 
However, it remains unknown whether there exists a $3$-graph $F$ with $\pi_{\rm u}(F)=\frac{1}{2}$. 
% It would be exciting if the lower bound in Inequality~\eqref{eq:complete hypergraph} is optimal, which would imply that $\Pi^{(2)} \subset \Pi_{\rm u}^{(3)}$. However, it should be pointed out that if this is true, it might be much more difficult to prove than $\pi_{\rm u}(K^{(3)}_4)=1/2$, as for $K^{(3)}_6$, there is a second, apparently sporadic, palette construction that yields the lower bound $\pi_{\rm u}(K^{(3)}_6)\ge \frac{3}{4}$ as well (see~\cite[Section 5.1]{k43minus-2}).
% In addition to the aforementioned results, 
% there are currently very few results on the uniform Tur\'an density of $3$-graphs.

In this work, we prove that $\Pi^{(2)} \subset \Pi_{\rm u}^{(3)}$ and $\left\{ (r-1)^2/{r^2}: r\ge 2\right\} \subset \Pi_{\rm u}^{(3)}$ by developing a novel method using the Turán numbers 
of digraphs to establish the existence of these values, building off the techniques of Lamaison~\cite{Ander}.

\subsection{Tur\'an problems in digraphs}
A \emph{digraph} (directed graph) is a pair $D=(V(D),A(D))$, where $V(D)$ is a vertex set and $A(D)$ is a set of ordered pairs of distinct vertices, called the \emph{arc} (or \emph{directed edge}) set. 
Given a digraph $D$, its \emph{Tur\'an number} and \emph{Tur\'an density} are defined analogously to the graph case as follows:
\[
\mathrm{ex}(n,D)
:= \max\{ |A(G)| : G \text{ is a $D$-free digraph on } n \text{ vertices} \}
\quad\text{and}\quad
\pi(D):=\lim_{n\to\infty}\frac{\mathrm{ex}(n,D)}{n^2-n}.
\]
The Tur\'an problem for digraphs is similar to that for graphs and has also attracted considerable attention. 
Given a digraph $D$, its \emph{underlying graph} is the simple graph $G$ with
\[
V(G)=V(D)\quad\text{and}\quad
\{u,v\}\in E(G)\ \text{whenever}\ uv\in A(D)\ \text{or}\ vu\in A(D).
\]
In 1970, Brown and Harary~\cite{brown1970extremal} extended Tur\'an's theorem~\cite{Tu-graph} to digraphs whose underlying graph is $K_r$, and revealed relationships between the corresponding Turán numbers and ${\rm ex}(n,K_r)$.
A \emph{tournament} on $r$ vertices, denoted by $T_r$, is an orientation of $K_r$, that is, the digraph obtained from $K_r$ by replacing each edge with exactly one arc.
Let $\overleftrightarrow{K_r}$ denote the complete digraph on $r$ vertices, that is, the digraph obtained from $K_r$ by replacing each edge with two arcs, one in each direction.
Brown and Harary~\cite{brown1970extremal} determined the Tur\'an numbers of all tournaments and complete digraphs.

\begin{theorem}[\cite{brown1970extremal}]\label{Thm:one-digraphs}
Given $r \ge 3$, let $T_r$ be any tournament on $r$ vertices and $\overleftrightarrow{K_r}$ the complete digraph on $r$ vertices. Then, for all $n \in \mathbb{N}$, we have
\[
{\rm ex}(n,T_r) = 2\,{\rm ex}(n,K_r) \quad\text{and}\quad {\rm ex}(n, \overleftrightarrow{K_r}) = \binom{n}{2} + {\rm ex}(n,K_r).
\]
\end{theorem}

In addition to tournaments and complete digraphs, Brown and Harary~\cite{brown1970extremal} also studied the sum of two tournaments whose structure lies between $T_r$ and $\overleftrightarrow{K_r}$. Specifically, we define this sum operation for digraphs as follows. Given two vertex-disjoint digraphs $D$ and $D'$, let $D\oplus D'$ be the \emph{sum} of $D$ and $D'$, that is, the digraph on $V(D)\cup V(D')$ obtained from $D$ and $D'$ by adding all possible arcs in both directions between $V(D)$ and $V(D')$, i.e.,
\[
A(D\oplus D') = A(D)\cup A(D') \cup (V(D)\times V(D')) \cup (V(D')\times V(D)).
\]

\begin{theorem}[\cite{brown1970extremal}]\label{Thm:two-digraphs}
Given $r\ge 6$, let $T_{r_1}$ and $T_{r_2}$ be any two tournaments with $r_1+r_2=r$ and $|V(T_{r_i})|=r_i$ for $i\in\{1,2\}$. Then, for every $n\in\mathbb N$,
\[
{\rm ex}(n,T_{r_1}\oplus T_{r_2}) = 2\,{\rm ex}(n,K_r).
\]
\end{theorem}

%\section{Preliminaries and main results}
\subsection{Our results}
In this paper, we show that uniform Tur\'an densities exhibit an interesting correspondence with Tur\'an numbers for digraphs. 
By combining~\cref{Thm:one-digraphs,Thm:two-digraphs} with Tur\'an's theorem~\cite{Tu-graph}, we obtain 
\[
\pi(T_{r})=\pi(T_{r_1}\oplus T_{r_2})=\pi(K_{r})=\frac{r-2}{r-1}
\quad\text{for } r=r_1+r_2.
\]
We begin by identifying such digraphs through the following definition.

\begin{definition}\label{def:good-digraphs}
For an integer $r\ge 3$, we say that an $r$-vertex digraph $D$ is \emph{$r$-good} if $
{\rm ex}(n,D)=2\,{\rm ex}(n,K_r)$ for all $n\in\mathbb N$. Throughout this paper, we write $\mathcal{F}_r=\{D: D \text{~is~an } r\text{-good~digraph} \}$.
\end{definition}

To state our results, we introduce two types of palettes associated with digraphs, which will remain fixed throughout the paper.

Let $r\ge 3$ and $D_r$ be an $r$-vertex digraph. We define:
\begin{itemize}
\item The \emph{left palette}
\[
\mathcal Q^{L}_{D_r}=(\mathcal C_1\incup\mathcal C_2,\,
\mathcal A(\mathcal Q^{L}_{D_r}))~
\]
where $\mathcal C_1$ and $\mathcal C_2$ are disjoint color sets with $|\mathcal C_1|=r$ and $\mathcal C_2=\{c_{ab}:(a,b)\in\mathcal C_1\times\mathcal C_1\}$.
Let $V(D_r)=\mathcal C_1$. The family of admissible triples is defined by
\[
\mathcal A(\mathcal Q^{L}_{D_r})
=\{(a,b,c_{ab})\in\mathcal C_1\times\mathcal C_1\times\mathcal C_2:(a,b)\in A(D_r)\}.
\]

\item The \emph{right palette}
\[
\mathcal Q^{R}_{D_r}=(\mathcal C_3\incup\mathcal C_4,\,
\mathcal A(\mathcal Q^{R}_{D_r})),
\]
where $\mathcal C_3$ and $\mathcal C_4$ are disjoint color sets with $|\mathcal C_3|=r$ and $\mathcal C_4=\{c_{ab}:(a,b)\in\mathcal C_3\times\mathcal C_3\}$.
Let $V(D_r)=\mathcal C_3$. The family of admissible triples is defined by
\[
\mathcal A(\mathcal Q^{R}_{D_r})
=\{(c_{ab},a,b)\in\mathcal C_4\times\mathcal C_3\times\mathcal C_3:(a,b)\in A(D_r)\}.
\]
\end{itemize}

Given two palettes $\mathcal P_1=(\mathcal C_1,\mathcal A_1)$ and $\mathcal P_2=(\mathcal C_2,\mathcal A_2)$, their \emph{union} $\mathcal P_1\incup\mathcal P_2$ is the palette with color set $\mathcal C_1\incup\mathcal C_2$ and admissible set $\mathcal A_1\incup\mathcal A_2$.
The following theorem provides the first verifiable condition ensuring that the uniform Tur\'an density of a $3$-graph equals $\frac{r-2}{r-1}$ for all $r\ge3$.

\begin{theorem}\label{Thm:=r-1/r}
For any $r\ge3$ and any (not necessarily distinct) digraphs $D_r,D'_r\in\mathcal F_r$, 
let $\mathcal Q_{D_r\incup D'_r}=\mathcal Q^{L}_{D_r}\incup\mathcal Q^{R}_{D'_r}$.
Then every $3$-graph $F$ that is $\mathcal Q_{D_r\incup D'_r}$-colorable but not $\mathcal Q_{r-1}$-colorable satisfies
\[
\pi_{\mathrm u}(F)=\frac{r-2}{r-1}.
\]
\end{theorem}

Given Theorem~\ref{Thm:=r-1/r}, it is natural to ask the following question:  
given two digraphs $D_r,D'_r \in \mathcal F_r$, does there exist a $3$-graph that is $\mathcal Q_{D_r \incup D'_r}$-colorable but not $\mathcal Q_{r-1}$-colorable? 
Let
\[
\mathcal F_r^*:=\{D_r\in \mathcal F_r :  \text{ the underlying graph of } D_r \text{ is } K_r\}.
\]
In this work, we show that if $D_r,D'_r \in \mathcal F_r^*$, then such $3$-graphs indeed exist.

\begin{theorem}\label{Thm:=r-1/r-exist}
For any $r\ge3$ and any (not necessarily distinct) digraphs $D_r,D'_r\in\mathcal F_r^*$, 
let $\mathcal Q_{D_r\incup D'_r}=\mathcal Q^{L}_{D_r}\incup\mathcal Q^{R}_{D'_r}$.  
There exists a $3$-graph $F$ that is $\mathcal Q_{D_r\incup D'_r}$-colorable but not $\mathcal Q_{r-1}$-colorable.
\end{theorem}

\cref{Thm:=r-1/r-exist} leads us to investigate which $r$-vertex digraphs belong to the set $F_r^*$. 
By~\cref{Thm:one-digraphs,Thm:two-digraphs}, we already know that every $r$-vertex tournament belongs to $\mathcal F_r^*$ for all $r\ge3$, and every $r$-vertex digraph that is the sum of two tournaments belongs to $\mathcal F_r^*$ for all $r\ge6$. 
The following theorem extends~\cref{Thm:two-digraphs} to the sum of three tournaments.  

\begin{theorem}\label{Thm:three-digraphs}
Let $r\ge 11$ and let $T_{r_1},T_{r_2},T_{r_3}$ be tournaments with $|V(T_{r_i})|=r_i$ for $i\in[3]$, where $0\le r_i\le r$ and $r_1+r_2+r_3=r$. Then $T_{r_1}\oplus T_{r_2}\oplus T_{r_3}\in \mathcal F_r^*$.
\end{theorem}

We make the following remarks on~\cref{Thm:three-digraphs}:
\begin{itemize}
\item[(1)] The condition $r\ge11$ in~\cref{Thm:three-digraphs} cannot be weakened to $r\ge10$.  
Let $D_{10}=C_3\oplus C_3\oplus T'_4$,  where $C_3$ is the \emph{directed cycle} on three vertices, i.e., $A(C_3)=\{uv,vw,wu\}$, and $T'_4$ is a tournament on four vertices that contains a $C_3$.
We claim that ${\rm ex}(n,D_{10})$ is strictly larger than $2{\rm ex}(n,K_{10})$, i.e., $D_{10}\notin \mathcal F_{10}^*$.  
Indeed, consider the digraph
\[
D_n
= T^{\star}_{\lfloor n/5\rfloor}\oplus T^{\star}_{\lfloor (n+1)/5\rfloor}
\oplus T^{\star}_{\lfloor (n+2)/5\rfloor}\oplus T^{\star}_{\lfloor (n+3)/5\rfloor}
\oplus T^{\star}_{\lfloor (n+4)/5\rfloor},
\]
where $T^{\star}_s$ denotes the transitive tournament on $s$ vertices. 
It is easy to see that $D_n$ is $D_{10}$-free, since $D_n$ does not contain $C_3\oplus C_3\oplus C_3$. 
Moreover, for large $n$ we have
\[
|A(D_n)|>n(n-1)-5\binom{(n+4)/5}{2}
=\frac{9}{10}n^2-\frac{3}{2}n
>\frac{8}{9}n^2
\ge 2\,{\rm ex}(n,K_{10}).
\]

\item[(2)] \cref{Thm:three-digraphs} implies \cref{Thm:one-digraphs} and \cref{Thm:two-digraphs} for all $r\ge 11$, since $r_i$ can choose zero in~\cref{Thm:three-digraphs}.  
\end{itemize}

We now return to the left palette $\mathcal Q^{L}_{D_r}$ and the right palette $\mathcal Q^{R}_{D_r}$ for some $D_r\in\mathcal F^*_r$. 
One easily checks that a $3$-graph $F$ is $\mathcal Q^{L}_{D_r}$-colorable if and only if $F$ is $\mathcal Q^{R}_{D_r}$-colorable.
A natural question is whether, in~\cref{Thm:=r-1/r-exist}, the requirement that $F$ be $\mathcal Q_{D_r\incup D'_r}$-colorable  can be weakened to $F$ is $\mathcal Q^{L}_{D_r}$-colorable (equivalently, $\mathcal Q^{R}_{D_r}$-colorable). 
The following theorem shows that the condition of being $\mathcal Q_{D_r\incup D'_r}$-colorable cannot be weakened and yields a sufficient condition ensuring that the uniform Tur\'an density of a $3$-graph equals $\bigl(\frac{r-2}{r-1}\bigr)^2$ for all $r\ge3$. 
% We say that a tournament $D$ is \emph{transitive} if for any three distinct vertices $x,y,z\in V(D)$, the presence of arcs $xy$ and $yz$ in $D$ implies the arc $xz$.
For each $r\ge2$, we define the palette
\begin{equation}\label{eq:square-low bound palette}
\mathcal Q^{2}_r = ([r], \mathcal A (\mathcal Q^{2}_r)), \text{ with } 
\mathcal A (\mathcal Q^{2}_r) = \{ (x,y,z) \in [r]^3 : x \ne y \text{ and } y \ne z \}.
\end{equation}

\begin{theorem}\label{Thm:=r-1/r-square}
Given $r\ge3$, 
let $T_r$ be a transitive tournament on $r$ vertices. Then every $3$-graph $F$ that is $\mathcal Q^{L}_{T_r}$-colorable (or equivalently $\mathcal Q^{R}_{T_r}$-colorable) but not $\mathcal Q^{2}_{r-1}$-colorable satisfies
\[
\pi_{\mathrm u}(F) = \left( \frac{r-2}{r-1} \right)^2.
\]
\end{theorem}

As mentioned above, $\pi_{\mathrm u}(K_4^{(3)-}) = 1/4$ (see~\cite{k43minus-1, k43minus-2}) and $\pi_{\mathrm u}(F^{\star}_5) = 1/4$ (see~\cite{chen2022beyond}) were proven using the hypergraph
regularity method. Now, by~\cref{Thm:=r-1/r-square}, it is straightforward to verify that $\pi_{\mathrm u}(K_4^{(3)-}) =\pi_{\mathrm u}(F^{\star}_5)= 1/4$.
Moreover, the next theorem guarantees the existence of such $3$-graphs for all $r\ge 3$, thus first confirming that the value $(\frac{r-2}{r-1})^2$ is attainable as a uniform Tur\'an density.

\begin{theorem}\label{Thm:=r-1/r-square-exist}
For any $r \ge 3$ and a transitive tournament $T_r$ be on $r$ vertices, there exist $3$-graphs $F$ that are $\mathcal Q^{L}_{T_r}$-colorable but not $\mathcal Q^{2}_{r-1}$-colorable.
\end{theorem}

As mentioned above, Reiher, R\"odl and Schacht~\cite{vanishing} characterized all $3$-graphs $F$ with $\pi_{\rm u}(F)=0$; namely, $\pi_{\rm u}(F)=0$ holds if and only if $F$ is $\mathcal Q^{-}_3$-colorable, where $\mathcal Q^{-}_3=([3],\{(1,2,3)\})$.  In this paper, we provide a very short proof for the following result.

\begin{theorem}\label{Thm:1/27-character}
Let 
\[\mathcal Q^{+1}_5=([5], \{(1,2,3),(4,5,1)\} ), ~\mathcal Q^{+2}_5=([5], \{(1,2,3),(4,1,5)\})\]
Then every $3$-graph $F$ that is $\mathcal Q^{+i}_5$-colorable for $i\in \{1,2\}$  but not $\mathcal Q^{-}_3$-colorable, satisfies  $\pi_{\rm u}(F)=1/27$. 
\end{theorem}

We note that the statement of \cref{Thm:1/27-character} is equivalent to~\cite[Theorem~15]{1/27}. In particular, Garbe, Kr\'al' and Lamaison~\cite{1/27} took about twenty pages to prove this result using the hypergraph regularity method and constructed some $3$-graphs that satisfy this condition.

Returning to \cref{Thm:1/27-character}, the assumption ``not $\mathcal Q^{-}_3$-colorable'' is necessary: the characterization of Reiher, R\"odl and Schacht~\cite{vanishing} implies that every $3$-graph $F$ with $\pi_{\rm u}(F)\ge 1/27$ fails to be $\mathcal Q^{-}_3$-colorable.
This raises the question of whether $\mathcal Q^{+1}_5$-colorability or $\mathcal Q^{+2}_5$-colorability is also necessary for a $3$-graph $F$ with $\pi_{\rm u}(F)=1/27$.
Recently, Lamaison~\cite{Ander} showed that there exist $3$-graphs $F$ with $\pi_{\rm u}(F)=1/4$ that are $\mathcal Q^{+1}_5$-colorable.
In this paper, we show that there also exist $3$-graphs $F$ with $\pi_{\rm u}(F)=4/27$ that are $\mathcal Q^{+2}_5$-colorable. More generally, we prove the following conclusion.

\begin{theorem}\label{Thm:4/27-character}
Let 
\[ \mathcal Q'^{-}_3=([3], \{(1, 3, 1), (1, 3, 2), (2, 3, 1),(2, 3, 2)\}).\]
Then every $3$-graph $F$ that is $\mathcal Q^{+2}_5$-colorable  but not $\mathcal Q'^{-}_3$-colorable, satisfies  $\pi_{\rm u}(F)=4/27$. 
\end{theorem}

We note that~\cref{Thm:4/27-character} remains valid if $\mathcal Q'^{-}_3$ is replaced by any palette of density $4/27$.
Moreover, we construct $3$-graphs that are $\mathcal Q^{+2}_5$-colorable but not $\mathcal Q'^{-}_3$-colorable.

\begin{theorem}\label{Thm:4/27-graph}
There exists $3$-graphs $F$ that are $\mathcal Q^{+2}_5$-colorable but not $\mathcal Q'^{-}_3$-colorable.
\end{theorem}

We emphasize that $\mathcal Q^{+2}_5$-colorability is not necessary for having uniform Tur\'an density $4/27$. 
As mentioned above, Buci{\'c}, Cooper, Kr\'al', Mohr and Munh\'a Correia~\cite{cycle-uniform} determined the uniform Tur\'an density of tight cycles $C^{(3)}_\ell$: it is $0$ when $3\mid \ell$ and equals $4/27$ when $3\nmid \ell$ (for $\ell\ge 5$).
However, one can check that $C_5^{(3)}$ is not $\mathcal Q^{+2}_5$-colorable.

The rest of the paper is organized as follows. 
In the next section, we solve the Tur\'an problem for the sum of three tournaments. As a new and independently interesting result in extremal digraph theory, which directly implies~\cref{Thm:three-digraphs}.
In~\cref{sec-upper-bounds}, we will give the proof of~\cref{Thm:=r-1/r},~\cref{Thm:=r-1/r-square},~\cref{Thm:1/27-character} and~\cref{Thm:4/27-character} using~\cref{Thm:ander}.
In~\cref{sec-examples}, we prove \cref{Thm:=r-1/r-exist},~\cref{Thm:=r-1/r-square-exist} and~\cref{Thm:4/27-graph} by constructing explicit examples. At last, we conclude our paper in~\cref{sec-conclude}.

\section{Tur\'an number of the sum of tournaments}
% In this section we prove~\cref{Thm:three-digraphs} by determining the Tur\'an number of the sum of three tournaments and the corresponding extremal digraphs. 
% We begin by collecting several notation and auxiliary results that will be used throughout the argument.

As is well known, Tur\'an~\cite{Tu-graph} determined the unique $K_r$-free extremal graph on $n$ vertices, which we denote by $T_{n,r}$ and call the \emph{Tur\'an graph}. The Tur\'an graph $T_{n,r}$ is a complete $(r-1)$-partite graph on $n$ vertices whose part sizes differ by at most one. Thus, if we write $n \equiv \alpha \pmod{r-1}$ with $0\le \alpha \le r-2$, then
\begin{equation}\label{eq:ex(n,k)}
    {\rm ex}(n,K_r)
    = |E(T_{n,r})|
    = \frac{r-2}{2(r-1)}\,n^2 - \frac{(r-1-\alpha)\alpha}{2(r-1)}.
\end{equation}
For convenience, set $f(n,r) = 2|E(T_{n,r})|$ for all $n,r\in \mathbb N$. 
Let $\overleftrightarrow{T_{n,r}}$ be the digraph that is obtained by replacing each edge of $T_{n,r}$ by replacing each undirected edge with bidirectional arcs.

\cref{Thm:three-digraphs} is an immediate consequence of the following statement.

\begin{theorem}\label{Thm:copy-three-digraphs}
Let $r\ge 11$ and let $D_{r_1}, D_{r_2}, D_{r_3}$ be tournaments with $|V(D_{r_i})|=r_i$ for $i\in[3]$, where $0\le r_i\le r$ and $r_1+r_2+r_3=r$. Then for every $n\in\mathbb N$,
\[
{\rm ex}(n,D_{r_1}\oplus D_{r_2}\oplus D_{r_3}) = f(n,r).
\]
Moreover, $\overleftrightarrow{T_{n,r}}$ is the unique $D_{r_1}\oplus D_{r_2}\oplus D_{r_3}$-free extremal digraph on $n$ vertices.   
\end{theorem}

The remainder of this section is devoted to the proof of~\cref{Thm:copy-three-digraphs}.

\subsection{Preparation}
As a direct consequence of~\eqref{eq:ex(n,k)}, we obtain the following useful proposition.

\begin{proposition}
\label{pro:relation}
For all $n_1,n_2,r\in\mathbb N$, let $\alpha_1$ and $\alpha_2$ denote the residues of $n_1$ and $n_2$ modulo $r-1$, respectively. Then
\[f(n_1 + n_2, r) - f(n_1, r) - f(n_2, r) - 2\cdot \frac{r-2}{r-1}\cdot n_1n_2 = \begin{cases} 
2\cdot \frac{\alpha_1\alpha_2}{r-1}, & \alpha_1 + \alpha_2 < r - 1, \\[6pt]
2\cdot \frac{(r - 1 - \alpha_1)(r - 1 - \alpha_2)}{r-1}, & \alpha_1 + \alpha_2 \geq r - 1.
\end{cases} 
\]
\end{proposition}

Note that there are only four non-isomorphic tournaments on at most three vertices.
For convenience, in this section we fix the notation that $T_1$ denotes the trivial tournament on one vertex, $T_2$ the unique tournament on two vertices, $T_3$ the transitive tournament on three vertices, and $C_3$ the directed cycle on three vertices.
For an $s$-vertex complete digraph $\overleftrightarrow{K_s}$, we may view it as the sum of $s$ copies of $T_1$. 
Hence, we can write
\[
\overleftrightarrow{K_s} = T^1_1 \oplus \cdots \oplus T^s_1,
\]
where each $T^i_1$ is a copy of $T_1$, which we refer to as the $i$th $T_1$-summand of $\overleftrightarrow{K_s}$.
More generally, we have the following proposition.

\begin{proposition}\label{pro:sums}
For every $s\in\mathbb N$, the sum $T^1_2\oplus\cdots\oplus T^s_2$ of $s$ copies of $T_2$ contains every tournament on at most $2s$ vertices, and the sum
$T^1_1\oplus T^1_2\oplus\cdots\oplus T^s_2$ of one copy of $T_1$ and $s$ copies of $T_2$ contains every tournament on at most $2s+1$ vertices.
\end{proposition}

\begin{proof}
Let $G:=T^1_2\oplus\cdots\oplus T^s_2$, and let $T$ be any tournament on $2s$ vertices. Partition $V(T)$ arbitrarily into $s$ unordered pairs $\{x_1,y_1\},\{x_2,y_2\},\dots,\{x_s,y_s\}$.
Write the $i$th $T_2$-summand of $G$ as having vertex set $\{a_i,b_i\}$, where exactly one of $(a_i, b_i)$ and $(b_i, a_i)$ is an arc.

Define an injective map $\phi:V(T)\to V(G)$ by mapping $\{x_i,y_i\}$ to $\{a_i,b_i\}$ in such a way that the unique arc inside the $i$th summand agrees with the orientation between $x_i$ and $y_i$ in $T$; this is always possible by swapping $a_i$ and $b_i$ if necessary.
Now consider any two vertices $u\in\{x_i,y_i\}$ and $v\in\{x_j,y_j\}$ with $i\neq j$. By definition of the sum operation, $G$ contains both arcs
$(\phi(u), \phi(v))$ and $(\phi(v),\phi(u))$. Hence, whichever direction the tournament $T$ prescribes between $u$ and $v$, the corresponding arc is present in $G$. Therefore, $G$ contains every tournament on $2s$ vertices.

For the odd case, let $G':=T_1\oplus G$, and let $T'$ be any tournament on $2s+1$ vertices. Map an arbitrary vertex of $T'$ to the unique vertex of the $T_1$-summand, and apply the previous argument to embed the remaining $2s$ vertices into $G$. Thus $G'$ contains every tournament on $2s+1$ vertices.
\end{proof}

We also need a simple structural fact about tournaments on six vertices.

\begin{lemma}\label{lemma:T6}
Every tournament on six vertices can be partitioned into two vertex-disjoint  $T_3$.
\end{lemma}

Before proving~\cref{lemma:T6}, we recall a classical upper bound on the number of copies of $C_3$ in a tournament.

\begin{theorem}[\cite{number-of-3-cycles}]\label{thm:moon-moser-3cycles}
Let $D$ be a tournament on $n$ vertices, and let $D(C_3)$ denote the number of copies of $C_3$ in $D$.  Then
\[
D(C_3) \le
\begin{cases}
\dfrac{n\,(n^2-4)}{24}, & n \text{ is even},\\[6pt]
\dfrac{n\,(n^2-1)}{24}, & n \text{ is odd}.
\end{cases}
\]
% Moreover, equality holds if and only if $T_n$ is a regular (or near-regular) tournament, i.e., each vertex has out-degree and in-degree differing by at most $1$.
\end{theorem}

\begin{proof}[Proof of~\cref{lemma:T6}]
Let $D$ be a tournament on six vertices. Every $3$-subset of $V(D)$ induces either a transitive tournament or a directed $3$-cycle. By~\cref{thm:moon-moser-3cycles}, the number of directed $3$-cycles in $D$ is at most $8$. Hence $D$ contains at least $\binom{6}{3}-8=12$ transitive tournaments.

Fix a transitive tournament $T_3$ in $D$. Each arc of $T_3$ lies in at most three transitive tournaments other than $T_3$ itself. Hence, there exists one that is vertex-disjoint from $T_3$. Let $T'_3$ be such a transitive tournament. Then $T$ and $T'$ form the desired partition of $V(D)$.
\end{proof}

For a digraph $D$ and a vertex $v\in V(D)$, let $d^+_D(v)$ be the {\it out-degree} of $v$, i.e., $d^+_D(v)=|\{u\in V(D): (v,u)\in A(D)\}|$, $d^-_D(v)$ be the {\it in-degree} of $v$, i.e., $d^-_D(v)=|\{u\in V(D): (u,v)\in A(D)\}|$, and  the {\it total degree} $d_D(v):=d^+_D(v)+d^+_D(v)$. 
Moreover, for a subset $U\subseteq V(D)$, the \emph{total connection} between a vertex $v$ and $U$ is
\[
e(v, U):=\bigl|\{(v, u)\in A(D): u\in U\}\bigr|+\bigl|\{(u,v)\in A(D):u\in U\}\bigr|,
\]
that is, the number of arcs with one endpoint $v$ and the other in $U$. If $U=\{u\}$ is a singleton, then we simply write $e(v, u)$ instead. 

\subsection{Proof of~\cref{Thm:copy-three-digraphs}} For the sake of clarity, we split the argument into two separate parts: one for the extremal value and the other for the uniqueness of the extremal digraph.

\begin{proof}[Proof of the extremal value.] 
Let $D:=D_{r_1}\oplus D_{r_2}\oplus D_{r_3}$, where $r=r_1+r_2+r_3$, and assume that $r\ge 11$ and $0\le r_i\le r$ for $i\in[3]$. 
For all $n\in \mathbb N$, we first prove that
\[
\ex(n,D)=f(n,r).
\]

If $n\le r-1$, then the statement is immediate, since $|A(\overleftrightarrow{K_{n}})|=f(n,r)$. Suppose, for a contradiction, that the statement fails, and let $G$ be a counterexample with the minimum possible number $n\ge r$ of vertices. Thus, $G$ is $D$-free and
\begin{equation}\label{eq: low bound of counterexample}
    |A(G)|\ge f(n,r) +1,
\end{equation}
while every digraph $G'$ with $|V(G')|<n$ and $|A(G')|>f (|V(G')|,r)$ contains a copy of $D$. Consequently, each vertex $v$ in $G$ satisfies
\begin{equation}\label{eq: vertex-low-degree}
e(v, V(G)\setminus \{v\}) \ge 2\frac{r-2}{r-1}(n-1)+1.
\end{equation}

Let $G_1$ be an $(r-1)$-vertex induced subdigraph of $G$ with the maximum possible number of arcs. 
We will first obtain some structural information about $G$ from the following claims. 

\begin{claim}\label{clm:G1to G2} For every $v\in V(G)\setminus V(G_1)$, $e(v, V(G_1))\le 2r-3$. Moreover,
there exists $v_1 \in V(G)\setminus V(G_1)$ such that $e(v_1, V(G_1))=2r-3$.
\end{claim}

\begin{proof}
If there exists $v' \in V(G)\setminus V(G_1)$ such that $e(v', V(G_1))=2r-2$. This implies that $G[V(G_1)\cup\{v'\}]$ would be the complete digraph $\overleftrightarrow{K_r}$, which contains $D$ as a subdigraph. Moreover, if $e(v, V(G_1))\le 2r-4$ for every $v\in V(G)\setminus V(G_1)$, then counting arcs induced on $V(G)\setminus V(G_1)$ and between $V(G_1)$ and $V(G)\setminus V(G_1)$ and using the minimality of $G$ together with~\cref{pro:relation} yields
\[
|A(G)|
\le |A(G_1)|+(2r-4)(n-r+1)+ f(n-r+1,r)
\le f(n,r),
\]
contradicting~\eqref{eq: low bound of counterexample}.
\end{proof}
 
Let $G_2:=G[V(G_1)\cup\{v_1\}]$. Then we have the following claim.
\begin{claim}\label{clm:G2-max}
Among all $r$-vertex subdigraphs of $G$, the subdigraph $G_2$ maximizes the number of arcs. Moreover, $G_2$ is a sum of $s_1$ copies of $T_1$ and $s_2$ copies of $T_2$ for some integer $0\le s_1\le 1$:
\[
G_2 = \underbrace{T_1\oplus\cdots\oplus T_1}_{s_1\text{ summands}}
\ \oplus\
\underbrace{T_2\oplus\cdots\oplus T_2}_{s_2\text{ summands}}, \qquad s_1+2s_2=r.
\]
\end{claim}
\begin{proof}
Let $G'$ be any $r$-vertex subdigraph of $G$. Since $G'$ is $D$-free, there exists $v' \in V(G')$ such that $e\bigl(v',\,V(G')\setminus\{v'\}\bigr) \le 2r-3$, and
\[
|A(G')|
= e\bigl(v',V(G')\setminus\{v\}\bigr)+|A(G'-v')|
\le (2r-3)+|A(G_1)|
=|A(G_2)|.
\]
Recalling the construction of $G_2$,  $e(v_1, V(G_2)\setminus\{v_1\})= 2r-3$. Thus, $v_1$ is bidirectionally adjacent to all but \emph{exactly one} vertex $u_1$ of $G_1$, and between $v_1$ and $u_1$ there is exactly one arc. By the maximality of $G_1$, $u_1$ is bidirectionally adjacent to all vertices in $V(G_1)\setminus\{u_1\}$. 
Consequently, $G_2$ decomposes as a sum of $T_1$'s and $T_2$'s. 
Moreover, by~\cref{pro:sums}, if $s_1\ge 2$ then $G_2$ contains a copy of $D$, a contradiction. Hence, $s_1\in\{0,1\}$.
\end{proof}

By~\cref{clm:G2-max}, we next split into two cases.

\medskip
\noindent\textbf{Case 1: $s_1 = 0$ (so $r$ is even).} 
Let 
\[
G_2 = T^1_2\oplus T^2_2\oplus \cdots\oplus T^{r/2}_2,
\]
where each $T^i_2$ is a copy of $T_2$ for $i\in [\frac{r}{2}]$.
Without loss of generality, assume that $r_1$ and $r_2$ are odd with $r_1\ge r_2$ and $r_3$ is even; if all three are even, then $G_2$ would already contain $D$ by~\cref{pro:sums}.

\begin{claim}\label{clm: vertex degree in G2} 
Every induced subgraph $G'_2$ of $G$ that is isomorphic to $G_2$ satisfies that $e(v, V(G'_2))\le 2r-2$ for each $v\in V(G)\setminus V(G'_2)$.
\end{claim}

\begin{proof}
Suppose that there is $v\in V(G)\setminus V(G'_2)$ such that $e(v, V(G'_2))=2r-1$. We can replace a suitable vertex of $G'_2$ by $v$ then which yields
an $r$-vertex subdigraph with more arcs than $G'_2$,  contradicting~\cref{clm:G2-max}.
\end{proof}
Let
\[
X:=\{v\in V(G)\setminus V(G_2): e(v,V(G_2))=2r-2\},
\qquad x:=|X|.
\]
Then we have
\[
|A(G)|\le f(n-r, r)+ x(2r-2)+ (n - r - x)(2r - 3)+|A(G_2)|.
\]
Note that $|A(G_2)|=r(r-1)-r/2$. By the low bound of $|A(G)|$ in~\eqref{eq: low bound of counterexample} and~\cref{pro:relation},
\begin{equation}\label{eq:x-lb}
x\ge \frac{2nr-6n-r^2+3r+2}{2(r-1)}.
\end{equation}
Note that for each $w\in X$, it satisfies the following two properties: 
\stepcounter{propcounter}
\begin{enumerate}[label = {{\rm (\Alph{propcounter}\arabic{enumi})}}]
\item \label{item:X1} By the maximality of $G_2$, there is no vertex $v\in V(G_2)$ such that $e(w, V(G_2)\setminus \{v\})=e(w, V(G_2))=2r-2$. Otherwise, we can obtain an $r$-vertex subdigraph with more arcs than $G_2$ by replacing $v$ by $w$.
\item \label{item:X2} By the maximality of $G_1$, there is no two vertices $u_1, u_2 \in V(G_2)$ such that $e(w, \{u_1\})=e(w, \{u_2\})=1$ but $e(\{u_1\}, \{u_2\})=2$. Otherwise, we can obtain an $(r-1)$-vertex subdigraph with more arcs than $G_1$ by considering $G[V(G_1)\setminus\{v_1, v_2\}\cup \{w\}]$.
\end{enumerate}

\begin{claim}\label{G3}
There exists an $(r+1)$-vertex induced subdigraph $G_3$ such that $G_3$ is a sum of exactly one $T_3$ and $\frac r2-1$ copies of $T_2$, that is,
\[
G_3 = T_3 \oplus T^1_2\oplus  \cdots\oplus T^{r/2-1}_2.
\]
\end{claim}

\begin{proof}
For a contradiction, assume that such $G_3$ does not exist in $G$. 
Combing~\ref{item:X1} and~\ref{item:X2}, for every $w\in X$, there must exist a unique $T_2$-summand $T^i_2$ of $G_2$ such that $w$ is connected to each of the two endpoints of $T^i_2$ by exactly one arc. In particular, $G[w\cup V(T^i_2)]$ forms a directed $3$-cycle. Therefore, we can partition $X$ into $\{X_1,\dots,X_{r/2}\}$ by defined
\[
X_i=\{w\in X: G[\{w\}\cup V(T^i_2)] \text{~is a directed } 3\text{-cycle in } G\}.
\]
Moreover,  if there are two distinct vertices  $w_1, w_2\in X$ such that them and a same $T_2$-summand of $G_2$ form a directed $3$-cycle, then $e(\{w_1\}, \{w_2\})=0$, otherwise $G[\{w_1, w_2\}\cup V(T_2)]$ will contain a $T_3$, yielding $G_3$. Therefore, each induced digraph $G[X_i]$ is empty for $i\in [\frac{r}{2}]$, which implies that 
\[
e(w, V(G)\setminus \{w\})\le 2(n-r-|X_1|)+2(r-1)=2(n-|X_1|-1)
\]
for each $w\in X_i$. 
Owing to~\eqref{eq: vertex-low-degree}, we have 
\begin{equation}\label{eq:xi}
|X_i|\le  \frac{n-r}{r-1}+ \frac{1}{2},    
\end{equation}
and hence,
\[
  x = \sum_{i=1}^{r/2} |X_i| \;\le\; \frac{r}{2}\cdot(\frac{n-r}{r-1}+\frac{1}{2})
            = \frac{rn-r^2}{2(r-1)}+\frac{r}{4},
\]
contradicting~\eqref{eq:x-lb} for all $n\ge r\ge 6$. Thus, $G_3$ must exist in $G$.
\end{proof}

Now fix such a subdigraph $G_3$, for convenience, write it as
\[
G_3 = T_3\oplus T^1_2\oplus \cdots\oplus T^{r/2-1}_2.
\]
Note that each $4$-vertex tournament contains a $T_3$ as a subdigraph. Hence, $T_3\oplus T^1_2\oplus \cdots\oplus T^{(r_1-3)/2}_2$ contains $D_1$ when $r_1\ge 4$. Meanwhile, By~\cref{pro:sums}, $ T^{(r_1-1)/2}_2\oplus \cdots\oplus T^{(r_1+r_2-2)/2}_2$ contains $D_2$   and $ T^{(r_1+r_2)/2}_2\oplus \cdots\oplus T^{r/2-1}_2$ contains $D_3$, which means that $G_3$ contains $D$, a contradiction.
Therefore, we have $r_1\le 3$. In this case, if exactly one of $D_1$ and $D_2$ is a transitive tournament $T_3$, then $G_3$ will also contain  $D$, a contradiction.
Thus, we can further assume $r_2\le r_1\le 3$ and both $D_1$ and $D_2$ are not a transitive tournament. Due to $r\ge 11$ and $r$ is even, we have $r_3\ge 6$. By~\cref{lemma:T6},  $D_3$ contains two vertex-disjoint $T_3$.
Since $G_3$ can be viewed as $G_2$ with an additional vertex $a$, it follows~\cref{clm: vertex degree in G2}
that each vertex $v\in V(G)\setminus V(G_3)$ must satisfy $e(v, V(G_3)\setminus\{a\})\le 2r-2$. Consequently, we have $e(v, V(G_3))\le 2r$. Define  
\[
  Y = \{\,v\in V(G)\setminus V(G_3): e(v, V(G_3)) = 2r\},
  \quad y = |Y|.
\]
Then we have
\[
|A(G)|\le f(n-(r+1), r)+ 2r\cdot y+ (2r - 1)(n - (r+1)-y)+|A(G_3)|.
\]
Note that $|A(G_3)|=(r+1)r-3-r/2+1=r^2-r/2-2$. By the low bound of $|A(G)|$ in~\eqref{eq: low bound of counterexample} and~\cref{pro:relation},
\[
  y \;\ge\; \frac{2nr - 10n - r^2 + 5r + 12}{2(r-1)}.
\]
Note that $G$ is $D$-free that forces each $w \in  Y$ must connect to $G_3$ via an arc to two endpoints of a $T_2$-summand and bidirectional arcs to all other vertices. In particular, $w$ and the $T_2$-summand connected by unidirectional arcs can not form a $T_3$; otherwise, $G[V(G_3)\cup \{w\}]$ also contain $D$, since $D_3$ contains two vertex-disjoint $T_3$.
Next, we partition $Y$ into $\{Y_1,\dots,Y_{r/2-1}\}$ by defined
\[
Y_i=\{w\in Y: G[\{w\}\cup V(T^i_2)] \text{~is a directed } 3\text{-cycle in } G\}.
\]
By an analysis analogous to~\eqref{eq:xi}, all vertex in the same $Y_i$ is non-adjacent in $G$ and every subset  $Y_i$ satisfies 
\[
|Y_i|\le \frac{n-r}{r-1}+ \frac{1}{2},
\]
and hence,
\[
  y =\sum_{i=1}^{r/2-1} |Y_i| \;<\; \frac{r-2}{2}\cdot (\frac{n-r}{r-1}+ \frac{1}{2})
       = \frac{nr - r^2-2n+2r}{2(r-1)}+\frac{r-2}{4},
\]
which is a contradiction with the low bound of $y$. This rules out \textbf{Case~1}.

\medskip
\noindent\textbf{Case 2: $s_1=1$ (so $r$ is odd).}
In this case, 
\[
G_2 =T_1\oplus T^1_2\oplus T^2_2\oplus \cdots\oplus T^{(r-1)/2}_2,
\]
where each $T^i_2$ is a copy of $T_2$ for $i\in [\frac{r-1}{2}]$. Without loss of generality, assume that $r_1,r_2,r_3$ are all odd with $r_1\ge r_2\ge r_3$; if there are two even in $r_1$, $r_2$ and $r_3$, then $G_2$ would already contain $D$ by~\cref{pro:sums}. Moreover, due to $r\ge 11$, $D_1$ contains a $T_3$. We distinguish the following two cases.

\medskip
\noindent\textbf{Case 2.1: There exists $v_2\in V(G)\setminus V(G_2)$ such that $e(v_2, V(G_2))=2r-1$.} By the maximality of $G_2$,  $v_2$ is only adjacent to $T_1$ by exactly one directed edge and bidirectionally adjacent to every other vertex of $G_2$. Let $G_4=G[\{v_2\}\cup V(G_2)]$, then we can write
\[
G_4 =T^1_2\oplus T^2_2\oplus \cdots\oplus T^{(r+1)/2}_2.
\]
A standard maximality argument (as in~\cref{clm:G2-max}) shows that $G_4$ maximizes the number of arcs among all $(r+1)$-vertex subdigraphs of $G$.
Moreover, if there exists a vertex $v \in V(G)\setminus V(G_4)$ such that $e(v, V(G_4)) = 2r + 1$, then replacing an appropriate vertex of $G_4$ by $v$ produces an $(r+1)$-vertex subdigraph 
with more edges than $G_4$, contradicting the maximality of $G_4$. Let
\[
  Z = \{\,v\in V(G)\setminus V(G_4): e(v, V(G_4)) = 2r\},
  \quad z = |Z|.
\]
Then we have
\[
|A(G)|\le f(n-(r+1), r)+ 2r\cdot z+ (2r - 1)(n - (r+1)-z)+|A(G_4)|.
\]
Note that $|A(G_4)|=(r+1)r-(r+1)/2$. By the low bound of $|A(G)|$ in~\eqref{eq: low bound of counterexample} and~\cref{pro:relation},
\begin{equation}\label{eq: zi}
    z \geq \frac{2nr - 10n - r^2 + 2r + 15}{2(r-1)}.
\end{equation}
Note that $G$ is $D$-free that forces each $w \in  Z$ must connect to $G_4$ via an arc to two endpoints of a $T_2$-summand and bidirectional arcs to all other vertices. Partition $Z$ into subsets $\{Z_1, \dots, Z_{(r+1)/2}\}$,  according to the unique $T_2$-summand that connected with $w$ by unidirectional arcs. Since $D_1$ contains $T_3$, all vertex in the same $Z_i$ is non-adjacent in $G$; otherwise, $G$ contains $D$. Analogous to the size of $X_i$, we have
\[
|Z_i| \le  \frac{n-r}{r-1}+ \frac{1}{2},
\]
and hence, we have 
\[
z = \sum_{i = 1}^{(r+1)/2}|Z_i| < \frac{r+1}{2}\cdot (\frac{n-r}{r-1}+ \frac{1}{2}) = \frac{nr + n - r^2 - r}{2(r - 1)}+ \frac{r+1}{4},
\]
which is a contradiction with~\eqref{eq: zi}. This rules out \textbf{Case~2.1}.

\medskip
\noindent\textbf{Case 2.2: Every $v\in V(G)\setminus V(G_2)$ satisfies $e(v, V(G_2))\le 2r-2$.} By the maximality condition of $G_2$, each vertex $v\in V(G)\setminus V(G_2)$ connected to $G_2$ with
exactly $2r-2$ arcs can only follow one of the following three connection patterns:
\stepcounter{propcounter}
\begin{enumerate}[label = {{\rm (\Alph{propcounter}\arabic{enumi})}}]
\item \label{item:pattern-1} $v$ has a unidirectional arc to the unique vertex of the $T_1$-summand and a further unidirectional arc to exactly one vertex of a $T_2$-summand;
\item \label{item:pattern-2} $v$ is nonadjacent to the $T_1$-summand and achieves $2r-2$ connections using only the $T_2$-summands;
\item \label{item:pattern-3} $v$ is adjacent to two unidirectional arcs (one to each vertex) in a single $T_2$-summand, where the $T_2$-summand and $v$ form a directed $3$-cycle $C_3$ in $G$. 
\end{enumerate}

For convenience, we split the proof into the following two subcases.

\medskip
\noindent\textbf{Case 2.2.1: There exists $b\in V(G)\setminus V(G_2)$ that follows~\ref{item:pattern-1}.}  In this case, let $G_5:=G[b \cup V(G_2)]$. 
Note that every $v$ in $V(G)\setminus V(G_5)$ satisfies $e(v, V(G_5))= e(v, V(G_2)) + e(v, \{b\})\le 2r$.
Let 
\[
 W= \{\,v\in V(G)\setminus V(G_5): e(v, V(G_5)) = 2r\}.
\]
Note that $|A(G_5)|=(r+1)r-\frac{r+3}{2}$.
The same counting as before (using~\cref{pro:relation} together with~\eqref{eq: low bound of counterexample}) yields
\[
|W| \geq  \frac{2nr  - 10n - r^2 + 4r + 13}{2(r-1)}.
\]

We now analyze the property of the vertex $b'$ in $W$.  Since $ e(b', V(G_2))=2r-2$ and  $e(b', \{b\})=2$, $b'$ is bidirectionally adjacent to $b$, and the connection pattern between $b'$ and $G_2$ follow one of the three patterns~\ref{item:pattern-1}--\ref{item:pattern-3}.  Through further analysis, we can obtain the following claims.  

\begin{claim}
 $b'$ and $G_2$ cannot follow~\ref{item:pattern-2}.
\end{claim}
\begin{proof}
Note that $b'$ and $G_2$ follows~\ref{item:pattern-2} means that $e(b', V(T_1))=0$ for the $T_1$-summand of $G_2$ and  $e(b', V(G_2)\setminus V(T_1))=2r-2$. In this case, we can replace the $T_1$-summand in $G_2$ with $b'$ to obtain an induced subgraph $G'_2$ that is  an isomorphic copy of $G_2$ such that $e(b, V(G'_2))=2r-1$, which case we have already ruled in {\bf Case 2.1}.
\end{proof}

\begin{claim}\label{clm:first pattern}
If $b'$ and $G_2$  follow~\ref{item:pattern-1}, then $b'$ can only connect via a unidirectional arc to the
other vertex in the $T_2$-summand that is connected to $b$ via a unidirectional arc.
\end{claim}

\begin{figure}[h]
\centering
\input{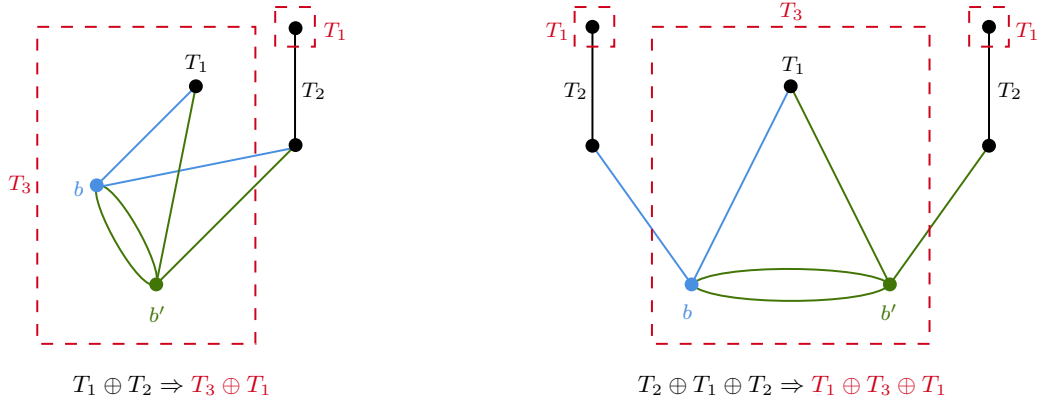}
\caption{{\small Construction of the proof of~\cref{clm:first pattern}.}}
\label{fig1}
\end{figure}

\begin{proof}
We first assume that $b$ and $b'$ have a unidirectional arc to the same vertex of some $T_2$-summand of $G_2$, then the digraph $G[V(G_2)\cup \{b,b'\}]$ contains the sum $T_3 \oplus T_1 \oplus T^2_2 \oplus \dots \oplus T^{(r-1)/2}_2$ (see~\cref{fig1}), which contain $D$, a contradiction.
In addition, if $b$ has a unidirectional arc to one vertex of a $T_2$-summand of $G_2$, and $b'$ has a unidirectional arc to one vertex of another $T_2$-summand of $G_2$. Then the digraph $G[V(G_2)\cup \{b,b'\}]$ contains the sum $T_3 \oplus T_1 \oplus T_1 \oplus T^3_2 \oplus \dots \oplus T^{(r-1)/2}_2$ (see~\cref{fig1}).  By~\cref{pro:sums},  $T_3 \oplus T_1 \oplus T_1 \oplus T^3_2 \oplus \dots \oplus T^{(r-1)/2}_2$ also contain $D$, a contradiction.
\end{proof}

Recalling that $G_2 =T_1\oplus T^1_2\oplus T^2_2\oplus \cdots\oplus T^{(r-1)/2}_2$.
Suppose without loss of generality that
$b$ has a unidirectional arc to the first $T_2$-summand of $G_2$. We partition $W$ into subsets as follows: $W_0$ consists of vertices that satisfy~\cref{clm:first pattern}, and 
$W_1,\dots, W_{(r-1)/2}$ correspond to vertices forming two unidirectional arcs with the $i$th $T_2$-summand of $G_2$. Then we have the following claim.

\begin{claim}\label{clm:Wi}
For each $i\in \{0,1, \dots, (r-1)/2\}$, the induced subdigraph $G[W_i]$ has no arcs.
\end{claim}

\begin{figure}[h]
\centering
\input{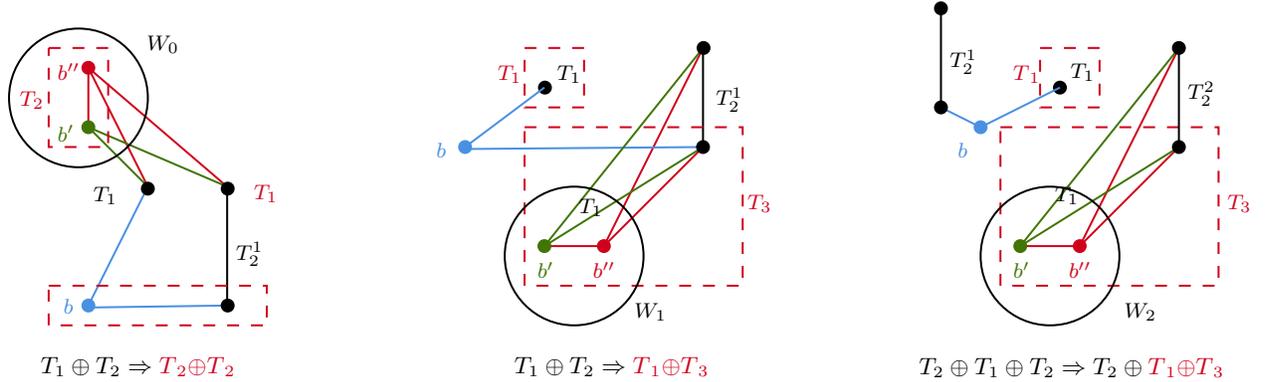}
\caption{{\small Construction of the proof of~\cref{clm:Wi}.}}
\label{fig2}
\end{figure}

\begin{proof}
If there are two vertices $b', b''\in W_0$ such that $e(b', b'')\neq 0$, $G_4$ would arise (see~\cref{fig2}), which has already been ruled out in the proof of {\bf Case 2.1}.
For $i\in [\frac{r-1}{2}]$, if there are two vertices $b', b''\in W_i$ such that $e(b', b'')\neq 0$, then the digraph $G[V(G_2)\cup \{b',b''\}]$ contains the sum $T_3 \oplus T_1 \oplus T^2_2 \oplus \dots \oplus T^{(r-1)/2}_2$ (see~\cref{fig2}), which contain $D$, a contradiction.
\end{proof}

By~\cref{clm:Wi} and an analysis analogous to~\eqref{eq:xi}, we have $|W_i| \leq \frac{n-r}{r-1}+\frac{1}{2}$, and hence 
\[
|W| =\sum_{i=0}^{(r-1)/2} |W_i| \;\le\; \frac{r+1}{2}\cdot (\frac{n-r}{r-1}+ \frac{1}{2})
 = \frac{nr + n - r^2 - r}{2(r - 1)}+ \frac{r+1}{4},
\]
which is a contradiction with the low bound of $|W|$.

\medskip
\noindent\textbf{Case 2.2.2: There does not exist a vertex that follows~\ref{item:pattern-1}.} 
In this case, every vertex $v\in V(G)\setminus V(G_2)$ with $e(v, V(G_2))=2r-2$ must attach to $G_2$ by one of the remaining two patterns. Let
\[
U = \{\,v\in V(G)\setminus V(G_2):\, e(v, V(G_2))=2r-2\}.
\]
Note that $|A(G_2)|=r(r-1)-(r-1)/2$ and every $v\in V(G)\setminus V(G_2)$ satisfies $e(v, V(G_2))\le 2r-2$.
The same counting as before (using~\cref{pro:relation} together with~\eqref{eq: low bound of counterexample}) yields
\[
|U| \geq  \frac{2nr-6n - r^2 + 2r + 3}{2(r-1)}.
\]
We partition $U$ into subsets as follows:
Let $U_0$ be the subset consisting of vertices attaching to $G_2$ by~\ref{item:pattern-2}. $U_1, \dots, U_{(r-1)/2}$ correspond to vertices forming two unidirectional arcs with the $i$th $T_2$-summand of $G_2$. By an analysis analogous to~\cref{clm:Wi}, each induced subdigraph $G[U_i]$ has also no arcs. 
Moreover, similar to~\eqref{eq:xi}, we have  $|U_i|\leq \frac{n-r}{r-1}+\frac{1}{2}$, and hence 
\[
|U| =\sum_{i=0}^{(r-1)/2} |U_i| \;\le\; \frac{r+1}{2}\cdot (\frac{n-r}{r-1}+ \frac{1}{2})
 = \frac{nr + n - r^2 - r}{2(r - 1)}+ \frac{r+1}{4},
\]
which is a contradiction with the low bound of $|U|$. 
This completes the proof of {\bf Case 2.2.2}.  
\end{proof}

\begin{proof}[Proof of the uniqueness of the extremal digraph.] 
Let $G$ be a $D$-free digraph on $n$ vertices with $|A(G)|=f(n,r)$. We will show $G = \overleftrightarrow{T_{n,r}}$.

We use induction on $n$. The assertion is trivial for $n\le r-1$, since $\overleftrightarrow{T_{n,r}}=\overleftrightarrow{K_{n}}$. 
Suppose that the statement fails for $n\ge r$. Among all digraphs that satisfy the hypothesis but not the conclusion, let $G$ be one for which has the minimum number $n\ge r$ of vertices.

Let $G_1$ be an induced subdigraph of $G$ on $r-1$ vertices with the maximum number of arcs. 
We next analyze the structure of $G$ by considering the $e(v, V(G_1))$ for vertices $v \in V(G)\setminus V(G_1)$.

\medskip
\noindent\textbf{Case 1:  For each $v \in V(G)\setminus V(G_1)$, $e(v, V(G_1)) \le 2r-4$.}
In this case, counting arcs induced on $V(G)\setminus V(G_1)$ and between $V(G_1)$ and $V(G)\setminus V(G_1)$ and using the structure of $G$ together with~\cref{pro:relation} yields
\[
f(n,r) =|A(G)|
\le |A(G_1)|+(2r-4)(n-r+1)+ f(n-r+1,r)
= f(n,r),
\]
where equality holds if and only if $G_1=\overleftrightarrow{K_{r-1}}$, and each $v\in V(G)\setminus V(G_1)$ satisfies $e(v, V(G_1))=2r-4$. There are two possible partitions of the arcs connecting  $v$ to $G_1$: either $v$ is connected by single arcs to two vertices $u$ and $u'$ of $G_1$, or by zero arcs to some vertex $u''$ of $G_1$. In the former case we can easily show that $G[\{v\}\cup V(G_1)]$ contains $D$. 
We conclude that, if label the vertices of $G_1$ as
\[
V(G_1)=\{v_1,v_2,\dots,v_{r-1}\},
\]
then $V(G)\setminus V(G_1)$  can be partitioned into sets $V_i$ consisting of those vertices not connected to $v_i$. Note that $G[V_i]$ is empty for each $i\in [r-1]$. Since $G_1$ is complete, the digraph $G$ is a complete $(r-1)$-partite digraph. Moreover, due to $|A(G)|=f(n,r)$, the sizes of the parts of $G$
differ by at most one, i.e., $G$ is $\overleftrightarrow{T_{n,r}}$.

\medskip
\noindent\textbf{Case 2:  There exists a vertex $v \in V(G)\setminus V(G_1)$ such that $e(v,V(G_1)) = 2r-3$.} 
When $n \ge 2r-2$, let $G_2 := G[V(G_1)\cup\{v\}]$. 
By the maximality of $G_1$, it follows that $G_2$ is an induced subdigraph of $G$ on $r$ vertices with the maximum number of arcs. The following argument is identical to the analysis of the proof of the extremal value, and hence yields a contradiction. 

Next, we consider $r \le n \le 2r-3$, and write $n = r-1+\alpha$, where $1 \le \alpha \le r-2$. By~\cref{pro:relation}, we have
\[
f(n,r) - f(n-1,r)
= 2\cdot\frac{r-2}{r-1}(n-1) + 2\cdot\frac{\alpha-1}{r-1}
= 2n-4.
\]

If there exists a vertex $v' \in V(G)$ such that  $d_G(v')= 2n-4$. 
Let $G' := G[V(G)\setminus\{v'\}]$.
Then $G'$ is an $(n-1)$-vertex $D$-free digraph with $|A(G')| = f(n-1,r)$.
By the induction hypothesis, $G'=\overleftrightarrow{T_{n-1,r}}$.
Since the vertex $v'$ has  the total degree $d_G(v')= 2n-4$, it is adjacent to all but exactly one part of $G'$, which follows that the structure of $G$ is $\overleftrightarrow{T_{n,r}}$.

We conclude that,  the total degree $d(v) \ge 2n-3 $ for every vertex $v \in V(G)$. Since $G$ has only $n$ vertices and $G$ is $D$-free,  the structure of $G$ must be a sum of $s_1$ copies of $T_1$ and $\frac{n-s_1}{2}$ copies of $T_2$ for some integer $0\le s_1\le 1$. Moreover, by the structure of $G$ and $G$ is $D$-free, we must have
\[
\frac{n}{2}\le s_1+\frac{n-s_1}{2}<\frac{r+2}{2}.
\]
Since $n = r-1+\alpha$, we obtain that $\alpha<3$.

We next compare the digraph $G$ with the complete digraph
$\overleftrightarrow{K_n}$ in terms of the number of arcs. Note that
\[
|A(\overleftrightarrow{K_n})| - |A(G)|=n(n-1)-f(r-1+\alpha,r)=2\alpha.
\]
By the structure of $G$, the number of $T_1$-summands of $G$ satisfies $n-4\alpha=s_1\le 1$, which means that
\[
\frac{r-2}{3}\le \alpha,
\]
that is a contradiction with the low bound of 
$\alpha$ when $r\ge 11$.
\end{proof}

%%%%%%%%%%%%%%%%%%%%%%%%%%%%%%%%%%%%%%%%%%%%%%%%%%%%%%%%%%%%%%%%%%%%%%%%%%%%%%%%%%%%%%%%
\section{Proof of uniform Tur\'an density theorems}\label{sec-upper-bounds}
This section is devoted to the proofs of ~\cref{Thm:=r-1/r},~\cref{Thm:=r-1/r-square},  ~\cref{Thm:1/27-character} and~\cref{Thm:4/27-character}. 
We first introduce some preliminary definitions and describe the main ideas behind our proofs.

We say that $\mathcal P=(\mathcal C (\mathcal P),\mathcal A (\mathcal P))$ is a \emph{subpalette} of $\mathcal Q=(\mathcal C(\mathcal Q),\mathcal A(\mathcal Q))$, and write $\mathcal P\subseteq \mathcal Q$, if there exists a mapping $f:\mathcal C(\mathcal P)\to \mathcal C(\mathcal Q)$ such that for every $(x,y,z)\in \mathcal A(\mathcal P)$ we have $(f(x),f(y),f(z))\in \mathcal A(\mathcal Q)$. 
An important observation of this relation is that, if 
$\mathcal P\subseteq \mathcal Q$, then each $3$-graph $F$ that is $\mathcal P$-colorable is also $\mathcal Q$-colorable.

Given a palette $\mathcal P=(\mathcal C,\mathcal A)$, its \emph{reverse}, denoted by ${\rm rev}(\mathcal P)$, is the palette $(\mathcal C,{\rm rev}(\mathcal A))$, where
${\rm rev}(\mathcal A):=\{(z,y,x):(x,y,z)\in \mathcal A\}$. An important property of this definition is that, a $3$-graph $F$ is $\mathcal P$-colorable if and only if $F$ is also ${\rm rev}(\mathcal P)$-colorable, by reversing
the order of the vertices of $F$.

The proofs of~\cref{Thm:=r-1/r,Thm:=r-1/r-square,Thm:1/27-character,Thm:4/27-character} utilize a novel method and could be formulated as one. 
More generally, to establish $\pi_{\rm u}(F)=d$ for a $3$-graph $F$ that is $\cp_1$-colorable but not $\cp_2$-colorable, 
using~\cref{Thm:ander}, our method splits the proofs into the following two parts:
\stepcounter{propcounter}
\begin{enumerate}[label = {{\rm (\Alph{propcounter}\arabic{enumi})}}]
\item \label{item:i} show that $d(\cp_2)\ge d$. Then we can obtain $\pi_{\rm u}(F)\ge d$;
\item \label{item:ii} for any palette $\cp$ with $d(\cp)> d$, show that either $\cp_1\subseteq \cp$ or ${\rm rev}(\mathcal P_1)\subseteq \cp$. Then we can obtain $\pi_{\rm u}(F)\le d$.
\end{enumerate}

Next, based on Steps~\ref{item:i} and~\ref{item:ii}, we first give the proof of~\cref{Thm:=r-1/r}.

\begin{proof}[Proof of ~\cref{Thm:=r-1/r}]
Given $r \ge 3$ and two digraphs $D_r, D'_r \in \mathcal F_r$, 
let $F$ be a $\mathcal Q_{D_r\incup D'_r}$-colorable but not $\mathcal Q_{r-1}$-colorable $3$-graph.  It is easy to calculate that $\mathcal Q_{r-1}$ has density $\frac{r-2}{r-1}$. 

Let $\mathcal P=(\cc, \ca)$ be a palette with $d(\mathcal P)>\frac{r-2}{r-1}$. We will show that $Q_{D_r\incup D'_r} \subseteq \cp$.
For palette $\mathcal P=(\cc, \ca)$, we define two auxiliary digraphs $G_L$ and $G_R$ on the vertex set $\cc$. Given $a, b \in \cc$, not necessarily distinct, we add $(a,b)$ in $G_L$ if there exists $c\in \cc$ such that $(a, b, c)\in\ca$. We add $(a,b)$ in $G_R$ if there exists a color $c\in \cc$ such  that $(c, a, b)\in\ca$. Note that $\cp$ has more than $\frac{r-2}{r-1}|\cc|^3$ admissible triples, which implies that $|A(G_L)|>\frac{r-2}{r-1}|\cc|^2$ and $|A(G_R)|>\frac{r-2}{r-1}|\cc|^2$.

When $|\mathcal C|\le r-1$, both $G_L$ and   $G_R$ have loops, since $|A(G_L)|>\frac{r-2}{r-1}|\cc|^2$ and $|A(G_R)|>\frac{r-2}{r-1}|\cc|^2$. By the construction of auxiliary digraphs,  $G_L$ and   $G_R$ have loops which implies that there exist  $(a, a, c), (c', a', a')\in \ca$. 
Recalling the definition of $Q_{D_r\incup D'_r}=(\mathcal C_1\incup\mathcal C_2\incup\mathcal C_3\incup\mathcal C_4, \mathcal A(\mathcal Q^{L}_{D_r}) \incup \mathcal A(\mathcal Q^{R}_{D'_r}))$,  there exists a mapping $f: \mathcal C_1\incup\mathcal C_2\incup\mathcal C_3\incup\mathcal C_4 \to \mathcal C$ such that $f(x_1)=a$ for each $x_1\in \mathcal C_1$, $f(x_2)=c$ for each $x_2\in \mathcal C_2$, $f(x_3)=a'$ each $x_3\in \mathcal C_3$, $f(x_4)=c'$ for each $x_4\in \mathcal C_4$.
Hence, in this case, we have $Q_{D_r\incup D'_r}\subseteq \cp$.

When $|\mathcal C|> r-1$, if $G_L$ and   $G_R$ also have loops, we are done. Now suppose that $G_L$ has no loops.
Since $|A(G_L)|>\frac{r-2}{r-1}|\cc|^2$ and $D_r\in \mathcal F_r$, $G_L$ contains a copy of $D_r$. In this case, we have $\mathcal Q^{L}_{D_r} \subseteq \cp$. The same argument applies to  $G_R$. Hence, we can always conclude that $Q_{D_r\incup D'_r}\subseteq \cp$.
\end{proof}

Before proving~\cref{Thm:=r-1/r-square}, we shall prove a digraph theoretic statement that will later be used in the proof of~\cref{Thm:=r-1/r-square}.

\begin{lemma}\label{lem:digraph-degree-square}
Given $r\ge 2$, let $T_r$ be a transitive tournament on $r$ vertices. If $D$ be a $T_r$-free digraph on $n$ vertices, then 
\begin{equation}
\max \left\{ \sum_{v\in V(D)}
d^+_{D}(v)^2,~\sum_{v\in V(D)}
d^-_{D}(v)^2 \right\}\leq \left(\frac{r-2}{r-1}\right)^2 n^3.
\end{equation}
\end{lemma}

\begin{proof}
%We prove the out-degree inequality, and the in-degree version follows by reversing all arcs.
We proceed by induction $r$ and $n$. The statement is clear for $r=2$ and each $n\in \mathbb N$, since $D$ is the empty digraph. Moreover, the statement holds for all $r\ge 3$ and $n\le r-1$. 

Now we assume that the statement holds for integers $r'<r$ and $n'<n$ with $r\ge 3$ and $n\ge r$. 
Let $w \in V(D)$ be a vertex with the maximum out-degree, i.e., $d^+_D(w) = \max\left\{d^+_D(v) : v \in V(D)\right\}$. Let $S$ be the set of out-neighbors of $w$ in $D$,  $D[S]$ be the digraph induced on $S$, and  $\overline{S} = V(D) \setminus S$.
We split the sum of out-degree squares into two parts: 
\[
\sum_{v \in V(D)} d^+_D(v)^2 = \sum_{v \in S} d^+_D(v)^2 + \sum_{v \in \overline{S}} d^+_D(v)^2.
\]
Next, we establish the upper bounds for these two parts separately.
Firstly, for any $v \in \overline{S}$, we have $d^+_D(v) \leq d^+_D(v)=|S|$. Hence, we have 
\[
\sum_{v \in \overline{S}} d^+_D(v)^2 \leq (n - |S|) \cdot |S|^2.
\]
For each $v \in S$, note that $d^+_D(v) \le d^+_{D[S]}(v) + n- |S|$. Hence, we have
\[
\sum_{v \in S} d^+_D(v)^2 \leq \sum_{v \in S} d^+_{D[S]}(v)^2 + 2(n - |S|)\sum_{v \in S} d^+_{D[S]}(v) + |S| \cdot (n - |S|)^2.
\]
Since $D$ is $T_r$-free, $D[S]$ is $T_{r-1}$-free. Otherwise, adding $w$ to a $T_{r-1}$ in $D[S]$ forms a $T_r$ in $D$, a contradiction. By~\cref{Thm:one-digraphs}, we have
\[
\sum_{v \in S} d^+_{D[S]}(v)= |A(D[S])|\le \frac{r-3}{r-2} |S|^2.
\]
Moreover, by the induction hypothesis, we have 
\[
\sum_{v \in S} d^+_{D[S]}(v)^2 \leq \left(\frac{r-3}{r-2}\right)^2 |S|^3.
\]
Substituting these into the inequality for $\sum_{v \in S} d^+_D(v)^2$, we obtain that 
\[
\sum_{v \in S} d^+_D(v)^2 \leq \left(\frac{r-3}{r-2}\right)^2 |S|^3 + 2(n - |S|)\frac{r-3}{r-2} |S|^2 + |S| \cdot (n - |S|)^2.
\]
Combining these two parts, the total sum satisfies that
\begin{equation}\label{eq:sum all v}
\sum_{v \in V(D)} d^+_D(v)^2 
\leq \left(\frac{r-3}{r-2}\right)^2 |S|^3 + 2(n - |S|)\frac{r-3}{r-2} |S|^2 + |S| \cdot (n - |S|)^2+ (n - |S|) \cdot |S|^2.    
\end{equation}
Set $\beta:=\frac{r-3}{r-2}$ and $t:=\frac{|S|}{n}$. We consider the following function 
\begin{align*}
g(t)=\beta^2 t^3 + 2\beta t^2(1-t) +t(1-t)^2+ t^2(1-t).
\end{align*}
When $t\in (0, 1]$, the maximum value of $g(t)$ occurs at $ t = \frac{r-2}{r-1}$. Therefore, we conclude that $g(t) \leq (\frac{r-2}{r-1})^2$ for all $t \in [0,1]$. By~\eqref{eq:sum all v}, we have
\[
\sum_{v \in V(D)} d^+_D(v)^2 
\leq g(t)\cdot n^3 \le \left(\frac{r-2}{r-1}\right)^2 n^3.
\]
The same argument applies to the in-degree version by considering the maximum in-degree vertex in $D$. 
% Hence, we obtain that
% \[
% \max \left\{ \sum_{v\in V(D)}
% d^+_{D}(v)^2,~\sum_{v\in V(D)}
% d^-_{D}(v)^2 \right\}\leq \left(\frac{r-2}{r-1}\right)^2 n^3.
% \]
\end{proof}

\begin{proof}[Proof of~\cref{Thm:=r-1/r-square}]
Given $r \ge 3$ and an $r$-vertex transitive tournament $T_r$, 
let $F$ be a $\mathcal Q^{L}_{T_r}$-colorable but not $\mathcal Q^2_r$-colorable $3$-graph.  It is easy to calculate that $\mathcal Q^2_r$ has density $(\frac{r-2}{r-1})^2$. 
		
Let $\mathcal P=(\cc, \ca)$ be a palette with $d(\mathcal P)>(\frac{r-2}{r-1})^2$. We will show that either $\mathcal Q^{L}_{T_r}\subseteq \cp$ or ${\rm rev}(\mathcal Q^{L}_{T_r})\subseteq \cp$.
Note that ${\rm rev}(\mathcal Q^{L}_{T_r})=\mathcal Q^{R}_{T_r}$. Hence, it suffices to show that $\mathcal Q^{L}_{T_r} \subseteq \cp$ or $\mathcal Q^{R}_{T_r} \subseteq \cp$.
		Similar to the proof of~\cref{Thm:=r-1/r}, 
		for the palette $\mathcal P=(\cc, \ca)$, we also define  two auxiliary digraphs $G_L$ and $G_R$ on the vertex set $\cc$. Given $a, b \in \cc$, not necessarily distinct, we add the arc $(a, b)$ in $G_L$ if there exists $c\in \cc$ such that $(a, b, c)\in\ca$. We add $(a, b)$ in $G_R$ if there exists a color $c\in \cc$ such  that $(c, a, b)\in\ca$. 
		
		When $|\mathcal C|\le r-1$, either $G_L$ contains a  loop or $G_R$ contains a  loop; otherwise, $(a,a,c)\notin \ca$ for any $a,c\in \cc$ and $(c',a',a')\notin \ca$ for any $a',c'\in \cc$ too, which implies that $d (\mathcal P)\le (\frac{|\mathcal C|-1}{|\mathcal C|})^2 \le (\frac{r-2}{r-1})^2$.  Recalling the definition of $\mathcal Q^{L}_{T_r}$, if $G_L$ has loops, we have $\mathcal Q^{L}_{T_r}\subseteq \cp$. Similarly, if $G_R$ has loops, then we have $\mathcal Q^{R}_{T_r}\subseteq \cp$. Hence, in this case, either $\mathcal Q^{L}_{T_r} \subseteq \cp$ or $\mathcal Q^{R}_{T_r} \subseteq \cp$.

When $|\mathcal C|> r-1$, if $G_L$ or  $G_R$ have loops, we are done. Moreover, if  $G_L$ contains  the transitive tournament $T_r$, then we have $\mathcal Q^{L}_{T_r}\subseteq \cp$. Similarly, if $G_R$ contains  the transitive tournament $T_r$,  then we have $\mathcal Q^{R}_{T_r}\subseteq \cp$. 
Now suppose that $G_L$ and $G_R$ are $T_r$-free and  have no loops. 
For any  $a\in \cc$, set $T_a=\{ (x, a, y)\in \ca: x,y\in \cc \}$. Note that
\[|T_a|\le d^{-}_{G_L}(a)\cdot d^{+}_{G_R}(a)\le \frac{1}{2}\left(d^{-}_{G_L}(a)^2+ d^{+}_{G_R}(a)^2\right),
\]
where the last inequality from the basic inequality.
By \cref{lem:digraph-degree-square}, we have
\[
\sum_{a\in \cc} d^{-}_{G_L}(a)^2\le \left(\frac{r-2}{r-1}\right)^2 |\cc|^3 ~\text{and~} \sum_{a\in \cc} d^{+}_{G_R}(a)^2\le \left(\frac{r-2}{r-1}\right)^2 |\cc|^3
\]
since  $G_L$ and $G_R$ are $T_r$-free. Hence, we have
\[
|\ca|=\sum_{a\in \cc}|T_a|\le \frac{1}{2} \sum_{a\in \cc} \left(d^{-}_{G_L}(a)^2+ d^{+}_{G_R}(a)^2\right)\le \left(\frac{r-2}{r-1}\right)^2 |\cc|^3,
\]
which is a contradiction.
\end{proof}

\begin{proof}[Proof of ~\cref{Thm:1/27-character}]
Let $F$ be a $3$-graph that is $\mathcal Q^{+i}_5$-colorable for $i\in \{1,2\}$  but not $\mathcal Q^{-}_3$-colorable. Since $d(\mathcal Q^-_3)=1/27$, we have $\pi_{{\rm u}}(F)\ge 1/27$. For the upper bound, it suffices to show that for any palette $\mathcal P= (\mathcal C, \mathcal A)$ with $d(\mathcal P)>1/27$, at least one of $\mathcal Q^{+1}_5$, $\mathcal Q^{+2}_5$ and ${\rm rev}(\mathcal Q^{+2}_5)$ is a subpalette $\mathcal P$.

We consider a palette $\mathcal P=(\mathcal C, \mathcal A)$ with $\mathcal Q^{+i}_5\nsubseteq \mathcal P$ for all $i \in [2]$. Let $L=\{c\in \mathcal C: \exists~(c, \alpha, \beta) \in \mathcal A \}$, $T=\{c\in \mathcal C: \exists~(\alpha, c,  \beta)\in \mathcal A \}$ and $R=\{c\in \mathcal C: \exists~(\alpha, \beta, c)\in \mathcal A\}$, where $\alpha,  \beta$ are not necessarily distinct. 
Then we claim that $L\cap T=L\cap R= T\cap R=\emptyset$. Indeed, $L\cap R\neq \emptyset$  means that there are $(c, \alpha_3, \beta_3), (\alpha_4, \beta_4,c) \in \mathcal A$, which implies that $\mathcal Q^{+1}_5\subseteq \mathcal P$. $L\cap T\neq \emptyset$  means that there are $(c, \alpha_1, \beta_1), (\alpha_2, c, \beta_2) \in \mathcal A$, which implies that $\mathcal Q^{+2}_5\subseteq \mathcal P$.  $T\cap R\neq \emptyset$  means that there are $( \alpha_5,c, \beta_5), (\alpha_6, \beta_6,c) \in \mathcal A$. In this case, we have ${\rm rev}(\mathcal Q^{+2}_5)\subseteq  P$.

Since $L\cap T=L\cap R= T\cap R=\emptyset$ and $|L|+|T|+|R|=|\mathcal C|$, we have 
\[
d(\mathcal P)=\frac{|\mathcal A|}{|\mathcal C|^3}\le \frac{|L|\cdot |T|\cdot |R|}{|\mathcal C|^3} \le \left( \frac{(|L|+|T|+|R|)/|\mathcal C|}{3}\right)^3=\frac{1}{27}.
\] 
\end{proof}

\begin{proof}[Proof of~\cref{Thm:4/27-character}]
Let $F$ be a $3$-graph that is $\mathcal Q^{+2}_5$-colorable but not $\mathcal Q^{-}_3$-colorable. 
Since $d(\mathcal Q'^-_3)=4/27$, we have $\pi_{{\rm u}}(F)\ge 4/27$.
	
Given a palette $\mathcal P=(\mathcal C, \mathcal A)$, let $L=\{c\in \mathcal C: \exists~(c, \alpha, \beta) \in \mathcal A \}$, $T=\{c\in \mathcal C: \exists~(\alpha, c,  \beta)\in \mathcal A \}$ and $R=\{c\in \mathcal C: \exists~(\alpha, \beta, c)\in \mathcal A\}$, where $\alpha,  \beta$ are not necessarily distinct. Suppose that $\mathcal Q^{+2}_5\nsubseteq \mathcal P$. Then we have $T\cap L=T\cap R=\emptyset$, which imples that
\[
d(P)=\frac{|\mathcal A|}{|\mathcal C|^3}\le \frac{|\mathcal C\setminus T| \cdot |T| \cdot |\mathcal C\setminus T|}{|\mathcal C|^3} \le t(1-t)^2\le 4/27,
\] 
where $t=|T|/\mathcal C|$ and $t\in (0,1)$. 
\end{proof}

\section{Proof of existence theorems}\label{sec-examples}
In this section, we will first prove~\cref{Thm:=r-1/r-exist} by using the following black box, which was recently developed by
Král', Kučerák, Lamaison, and Tardos~\cite{KKLT-25}. 

\begin{lemma}[\cite{KKLT-25}]\label{lem:classification palette}
Let $\mathcal P_1$ and $\mathcal P_2$ be two palettes. Then there exists a $3$-graph $F$ that is $\mathcal P_1$-colorable but not $\mathcal P_2$-colorable if and only if both $\mathcal P_1 \nsubseteq \mathcal P_2$ and $\mathcal P_1 \nsubseteq {\rm rev}(\mathcal P_2)$ hold.
\end{lemma}

\begin{proof}[Proof of~\cref{Thm:=r-1/r-exist}]
Given $r \ge 3$ and two digraphs $D_r, D'_r \in \mathcal F^*_r$, let $\cp_1=\mathcal Q_{D_r\incup D'_r}$ and $\cp_2=\mathcal Q_{r-1}$. By~\cref{lem:classification palette},  it suffices to show that $\mathcal P_1 \nsubseteq \mathcal P_2$ and $\mathcal P_1 \nsubseteq {\rm rev}(\mathcal P_2)$. 

Assume for contradiction
that $\mathcal P_1 \subseteq \mathcal P_2$, which means that there exists a mapping $f: \mathcal C_1\incup\mathcal C_2\incup\mathcal C_3\incup\mathcal C_4 \to [r-1]$ such that for every $(x,y,z)\in \mathcal A(\mathcal P_1)$ we have $(f(x),f(y),f(z))\in \mathcal A(\mathcal P_2)$. Due to $|\mathcal C_1|=r$, there must exist two distinct color $a, b\in \mathcal C_1$ such that $f(a)=f(b)$. Furthermore, due to $D_r\in \mathcal F^*_r$, either $(a, b, c_{ab})\in \mathcal(\mathcal P_1)$ or  $(b, a, c_{ba})\in \mathcal(\mathcal P_1)$. Let $\alpha=f(a)=f(b)$. Then $(\alpha, \alpha, \beta)\in \mathcal A(\mathcal P_2)$ for some $\beta\in \mathcal C(\mathcal P_2)$, which is impossible since $A(\mathcal P_2)=\{(x, y, z)\in [r-1]^3: x\neq y\}$.
Similarly, we show that $\mathcal P_1 \nsubseteq {\rm rev}(\mathcal P_2)$. Note that ${\rm rev}(\mathcal P_2)=([r-1], \{(z, y, x)\in [r-1]^3: x\neq y \})$. Due to $D'_r\in \mathcal F^*_r$, either $( c_{ab}, a, b)\in \mathcal A (\mathcal P_1)$ or  $(c_{ba}, b, a)\in \mathcal A (\mathcal P_1)$. If $\mathcal P_1 \subseteq {\rm rev}(\mathcal P_2)$, which implies that there exists $(\beta, \alpha, \alpha)\in \mathcal A(\mathcal P_2)$ for some $\beta, \alpha \in \mathcal C({\rm rev}(\mathcal P_2))$, producing a contradiction.

\end{proof}

To make the presentation more intuitive, in~\cref{Thm:=r-1/r-exist}, considering different pairs $(D_r, D'_r) \in \mathcal F^*_r \times \mathcal F^*_r$ allows us to obtain distinct $3$-graphs $F$. In the proofs of~\cref{Thm:=r-1/r-square-exist} and~\cref{Thm:4/27-graph} below, we will construct $3$-graphs that satisfy the conditions of the theorems, as opposed to relying on~\cref{lem:classification palette}.
In the construction we will use the following lemma for the linear hypergraphs.
Recall that a hypergraph is linear if any two edges have at most one vertex in common.

\begin{lemma}[\cite{Ander}]\label{lem:linear-graph}
For every $k\ge 3$, there exists a positive integer $n$ and a linear $k$-graph $H$ on the vertex set $[n]$ such that for every permutation $\sigma$ of $[n]$ there exists an edge $e$ in $H$ such that $\sigma$ is monotone on the vertices of $e$.
\end{lemma}

\begin{proof}[Proof of~\cref{Thm:=r-1/r-square-exist}]
Given $r\ge 3$, let $k=r+1$. By~\cref{lem:linear-graph}, there exists a positive integer $n$ and a linear $k$-graph $H$ on the vertex set $[n]$ such that for every permutation $\sigma$ of $[n]$ there exists an edge $e$ in $H$ such that $\sigma$ is monotone on the vertices of $e$.
Without loss of generality, assume that $A(T_r)=\{(i,j): 1\le i<j\le r\}$. We construct a $3$-graph $F_H$ on the vertex set $[n]$ such that $F_H$ is $\mathcal Q^L_{T_r}$-colorable but not $\mathcal Q^2_r$-colorable as follows.
For every edge $e \in E(H)$ with $e=\{x_0, x_1, \dots, x_r\}$ and $x_0<x_1<\dots<x_r$, place $\binom{r}{2}$ edges in $F_H$ on the following $3$-sets
\begin{itemize}
\item $\{x_0, x_i, y_j\}$ for any  $1\le i<j\le r$.
\end{itemize}

We first claim that $F_H$ is $\mathcal Q^L_{T_r}$-colorable. Recalling that $\mathcal Q^L_{T_r}$ has the color set $\cc_1 \incup \cc_2$ and the admissible set $\{(i, j, c_{ij})\in \cc_1 \times \cc_1 \times \cc_2:  (i, j)\in A(T_r)\}$.
Since the hypergraph $H$ is linear, the only pairs of edges of $F_H$ that intersect
in two vertices are those that are generated by the same edge of $H$, otherwise the intersection size is at most $1$. Moreover, in the natural order of $[n]$, for each edge $e \in E(H)$ with $e=\{x_0, x_1, \dots, x_r\}$,
consider the function $\varphi: \binom{[n]}{2}\to \cc_1 \incup \cc_2$ such that every edge $\{x_0, x_i, y_j\}$ with $1\le i<j\le r$ satisfies
\begin{itemize}
\item 
 $(\varphi(x_0x_i),\varphi(x_0x_j), \varphi(x_ix_j)) =( i, j,c_{ij})\in \mathcal A(\mathcal Q^L_{T_r})$.
\end{itemize}
Hence, $F_H$ is $\mathcal Q^L_{T_r}$-colorable. 

We next show that $F_H$ is not $\mathcal Q^2_{r-1}$-colorable.
Recalling that $\mathcal Q^2_{r-1} = ([r-1], \{ (x, y, z) \in [r-1]^3 : x \ne y \text{ and } y\ne z\}$.
Assume for contradiction that $F_{H}$ is $\mathcal Q^2_{r-1}$-colorable with a permutation $\sigma$ of $[n]$ and the function $\varphi: \binom{[n]}{2}\to \mathcal C(\mathcal Q^2_{r-1})$. By the property of $H$, there must exist an edge $e_H\in E(H)$ such that $\sigma$ is monotone on the vertices of $e_H$. Let $x_0<x_1<x_2<\dots<x_r$ be the vertices of $e_H$. 
Then there must be $\varphi(x_0x_i)= \varphi(x_0x_j)$ for some distinct $i,j \in [r]$.
If $\sigma$ is increasing on $e_H$, then $(\varphi(x_0x_i),\varphi(x_0x_j), \varphi(x_ix_j))\in \mathcal A(\mathcal Q^2_{r-1})$ which is impossible since $x\neq y$ for each $(x,y,z)\in  \mathcal A(\mathcal Q^2_{r-1})$. 
If $\sigma$  is decreasing on $e_H$, then $(\varphi(x_ix_j),\varphi(x_0x_i), \varphi(x_0x_j))\in \mathcal A(\mathcal Q^2_{r-1})$ which is also impossible since $y\neq z$ for each $(x,y,z)\in  \mathcal A(\mathcal Q^2_{r-1})$.
\end{proof}

\begin{proof}[Proof of~\cref{Thm:4/27-graph}]
Let $H$ be a linear $4$-graph on $[n]$ as in~\cref{lem:linear-graph}. For every edge $e_H$ with vertices $v_1 < v_2 < v_3 < v_4$, place an edge  on the triple $v_1v_{2}v_{3}$ and an edge  on the triple $v_1v_{3}v_{4}$, and to create a $3$-graph $F^{4}_{H}$.

We first claim that $F^{4}_{H}$ is  $\mathcal Q^{+2}_5$-colorable. Consider the function $\varphi: \binom{[n]}{2} \to \mathcal [5]$ such that for each $e_H\in E(H)$ with vertices $v_1 < v_2 < v_3 < v_4$ in the natural order of $[n]$ the following hold: 
\[
(\varphi(v_1v_2),\varphi(v_1v_3), \varphi(v_2v_3))=(4, 1, 5)\text{~and~} (\varphi(v_1v_3),\varphi(v_1v_4), \varphi(v_3v_4))=(1, 2, 3).
\] 
Since $H$ is linear, this coloring is consistent, and $F^{4}_{H}$ is $\mathcal Q^{+2}_5$-colorable.

We next show that $F^{4}_{H}$ is not $\mathcal Q'^{-}_3$-colorable. Let $\sigma$ be any order of $V(F^{4}_{H})$. By the property of $H$, there must exist an edge $e_H\in E(H)$ such that $\sigma$ is monotone on the vertices of $e_H$. Let $v_1 < v_2 < v_3< v_4$ be the vertices of $e_H$. For any function $\varphi: \binom{[n]}{2} \to \mathcal [5]$,   if $\sigma^*$ is increasing on $e_H$, then $(\varphi(v_1v_2),\varphi(v_1v_3), \varphi(v_2v_3))$ and  $(\varphi(v_1v_3),\varphi(v_1v_4), \varphi(v_3v_4))$ must belong to  $\mathcal A(\mathcal Q'^{-}_3)$, which is impossible since no color in $\mathcal C(\mathcal Q'^{-}_3)$ is both a top color and a left color. If $\sigma$ is decreasing on $e_H$,  then $(\varphi(v_3v_2),\varphi(v_3v_1), \varphi(v_2v_1))$ and  $(\varphi(v_4v_3),\varphi(v_4v_1), \varphi(v_3v_1))$ must belong to $\mathcal A(\mathcal Q'^{-}_3)$, which also is impossible since no color in $\mathcal C(\mathcal Q'^{-}_3)$ is both a top color and a right color. 
Thus, $F^{4}_{H}$ is a $3$-graph that is $\mathcal Q^{+2}_5$-colorable but not $\mathcal Q'^{-}_3$-colorable. 
\end{proof}

\section{Concluding remarks}\label{sec-conclude}
In this paper, we studied the uniform Tur\'an densities of $3$-graphs by combining several extremal results of digraphs.
We summarize our contributions as follows.

\cref{Thm:=r-1/r,Thm:=r-1/r-square,Thm:1/27-character,Thm:4/27-character} provide sufficient characterizations for the $3$-graphs $F$ whose uniform Turán densities take the values $\frac{r-2}{r-1}$, $(\frac{r-2}{r-1})^2$,  $1/27$ and $4/27$, respectively. 
\cref{Thm:=r-1/r-exist,Thm:=r-1/r-square-exist,Thm:4/27-graph} guarantee the existence of such $3$-graphs that satisfy the corresponding characterizations. 

Let $d\in \{\frac{r-2}{r-1}, (\frac{r-2}{r-1})^2: r\ge 3\} \cup \{1/27, 4/27\}$, such 
a sufficient characterization for $\piu (F)=d$ is outlined as follows:

\begin{enumerate}
\item $F$ is $\cp_1$-colorable but not $\cp_2$-colorable, where $\cp_2$ satisfies $d(\cp_2)= d$ and $\cp_1$ satisfies that for any palette $\cp$ with $d(\cp)> d$,  either $\cp_1\subseteq \cp$ or ${\rm rev}(\mathcal P_1)\subseteq \cp$.
\end{enumerate}

While this proof strategy appears to enable the search for new values of the uniform Turán density, the main challenge is finding suitable  $\cp_1$ and $\cp_2$.  Specifically, we need to show that every palette $\cp$ with $d(\cp)> d(\cp_2)$ should either have $\cp_1\subseteq \cp$ or ${\rm rev}(\mathcal P_1)\subseteq \cp$. Moreover, by~\cref{lem:classification palette}, both $\mathcal P_1 \nsubseteq \mathcal P_2$ and $\mathcal P_1 \nsubseteq {\rm rev}(\mathcal P_2)$ must hold, otherwise, such $3$-graphs $F$ do not exist.

For any $r\ge3$ and any (not necessarily distinct) digraphs $D_r, D'_r\in\mathcal F^*_r$, we prove the existence of $3$-graphs satisfying $\mathcal Q_{D_r\incup D'_r}$-colorable but not $\mathcal Q_{r-1}$-colorable (see~\cref{Thm:=r-1/r-exist}), using the black box~\cref{lem:classification palette}. 
In fact, if we know the structure of digraphs $D_r$ and $ D'_r$, we can construct a linear hypergraph first and then the $3$-graph $F$, following an approach analogous to the proof of~\cref{Thm:=r-1/r-square-exist}. 
For example, let $T_3$ be the transitive tournament on three vertices and consider $D_3=D'_3=T_3$. We can find a linear $7$-graph $H$ on the vertex set $[n]$ such that for every permutation $\sigma$ of $[n]$ there exists an edge $e$ in $H$ such that $\sigma$ is monotone on the vertices of $e$. Then for every edge $e_H$ with vertices $v_1 < v_2 < \dots < v_7$, place an edge on the triples $v_1v_2v_4$, $v_2v_3v_4$, $v_1v_3v_4$, $v_4v_5v_6$, $v_4v_5v_7$ and $v_4v_6v_7$  to create a $3$-graph $F_H$. Similar to the proof of~\cref{Thm:=r-1/r-square-exist}, we can check that such a $3$-graph $F_H$ is $\mathcal Q_{T_3\incup T_3}$-colorable but not $\mathcal Q_{2}$-colorable.
Consequently, although the $3$-graphs corresponding to different digraphs $D_r$ and $D'_r$ share the same uniform Turán density, their structures can be non-isomorphic. Thus, finding more digraphs in $\mathcal F^*_r$ allows us to determine more $3$-graphs $F$ with $\piu(F)=\frac{r-2}{r-1}$.

For each $k\in\mathbb N$, let $T_{r_1},\dots,T_{r_k}$ be tournaments with $\sum_{i=1}^{k} r_i=r$ and $|V(T_{r_i})|=r_i$. We show that  $T_{r_1}\oplus T_{r_2} \oplus T_{r_3}\in \mathcal F_r^*$ for all $r\ge 11$ (see~\cref{Thm:three-digraphs}). We believe~\cref{Thm:three-digraphs} can be extended to the sum of four or five tournaments by applying our method with more complex analysis. A natural question arises: 

\begin{question}\label{Q:many tournaments} 
For each $k\in\mathbb N$, does there exist a function  $f(k) \in \mathbb{N}$ such that for all $r\ge f(k)$, $T_{r_1}\oplus \dots \oplus T_{r_k}\in \mathcal F_r^*$?
\end{question}

Let $f^*(k)$ be the smallest integer $f(k)$ for which Question~\ref{Q:many tournaments} holds. 
Note that the complete digraph $\overleftrightarrow{K_r}$ can be viewed as the sum of $r$ tournaments of order~$1$. 
However, $\overleftrightarrow{K_r}\notin \mathcal F^*_r$ for any $r\ge3$. Thus, $f^*(k)>k$ for all $k\ge 2$.
By~\cref{Thm:one-digraphs}, we have $f^*(1)=3$. 
By~\cref{Thm:two-digraphs}, we obtain $f^*(2)\le 6$. 
Finally, by~\cref{Thm:three-digraphs} and the discussion following it, we have $f^*(3)=11$. So it would be very interesting to determine $f(k)$ for all $k\ge 4$.

Given a digraph $D$, we have
\[
|A(D)|=\sum_{v\in V(D)}d^+_{D}(v)= \sum_{v\in V(D)}d^-_{D}(v).
\]
The sum of the out-degrees (in-degrees) of $D$ can be regarded as the $\ell_1$-norm of its out-degree (in-degrees) vector.
Therefore, given a digraph $F$,  determining $\ex (n, F)$ is equivalent to determining the maximum $\ell_1$-norm of the out-degrees vector (or in-degrees vector) of an $F$-free $n$-vertex digraph.

Recalling that \cref{lem:digraph-degree-square} plays a key role in the proof of~\cref{Thm:=r-1/r-square},  we observe that \cref{lem:digraph-degree-square} determines the maximum $\ell_2$-norm of the out-degrees vector and in-degrees vector of all $T_r$-free digraphs, where $T_r$ is the transitive tournament on $r$ vertices. We believe that \cref{lem:digraph-degree-square} holds for all digraphs in the family $\mathcal F^*_r$. If \cref{lem:digraph-degree-square} can be extended to all $D\in \mathcal F^*_r$, then \cref{Thm:=r-1/r-square} and \cref{Thm:=r-1/r-square-exist} can also be extended accordingly to hold for $\mathcal F^*_r$.

We refer to the square of
the $\ell_2$-norm of the out-degrees vector (or in-degrees vector) as the out-degrees (or in-degree) squared sum, denoted by $\gamma^+_2(D)$ (or $\gamma^-_2(D)$):
\[
\gamma^+_2(D)=\sum_{v\in V(D)}d^+_{D}(v)^2,~\text{and }~ \gamma^-_2(D)=\sum_{v\in V(D)}d^-_{D}(v)^2.
\]
Moreover, given a digraph $D$, let $\gamma_2(D):=\max\{\gamma^+_2(D), \gamma^-_2(D)\}$ be the {\it degree squared sum} of $D$.
\begin{question}\label{Q:l2-norm}
Given a digraph $F$, what is the maximum degree squared sum
that an $F$-free $n$-vertex digraph $D$ can have?
\end{question}

It is worth noting that hypergraph Tur\'an problems in $\ell_2$-norm have seen some notable research progress (see~\cite{Balogh-norm-22-1,Balogh-norm-22-2}). Therefore, \cref{Q:l2-norm} is an interesting question that warrants investigation in its own right.

%Note that elements in the palettes are some ordered triples and digraph elements are some ordered pairs. Therefore, some extremal results in digraphs offer a viable pathway to determine the uniform Turán density of $3$-graphs. 

\bibliographystyle{abbrv}
\bibliography{ref}
\end{document}